\documentclass[11pt]{amsproc}
\usepackage{mathrsfs}
\usepackage{stmaryrd}
\usepackage{cases}
\usepackage{amssymb}
\usepackage{amsmath}
\usepackage{amsfonts}
\usepackage{graphicx}
\usepackage{amsmath,amstext,amsbsy,amssymb, color}

\usepackage{latexsym}
\usepackage{amsmath}
\usepackage{amssymb}
\usepackage{bm}
\usepackage{graphicx}
\usepackage{wrapfig}
\usepackage{fancybox}
\usepackage{cancel}
\newtheorem{theorem}{Theorem}[section]
\newtheorem{lemma}[theorem]{Lemma}
\newtheorem{proposition}[theorem]{Proposition}
\newtheorem{corollary}[theorem]{Corollary}
\theoremstyle{definition}

\theoremstyle{remark}
\newtheorem{remark}[theorem]{Remark}

\numberwithin{equation}{section} \errorcontextlines=0

\newcommand{\diag}{\mbox{diag}}

\newcommand{\cdet}{\mathrm{cdet}}
\newcommand{\rdet}{\mathrm{rdet}}

\newcommand{\cper}{\mathrm{cper}}
\newcommand{\rper}{\mathrm{rper}}

\newcommand{\ot}{\otimes}

\newcommand{\si}{\sigma}

\newcommand{\gl}{\mathfrak{gl}}

\newcommand{\qdet}{\mathrm{qdet}}

\begin{document}

\title[Quantum algebra of multiparameter Manin matrices]
{Quantum algebra of multiparameter Manin matrices}

\author{Naihuan Jing}
\address{Department of Mathematics, North Carolina State University, Raleigh, NC 27695, USA}
\email{jing@ncsu.edu}

\author{Yinlong Liu}
\address{School of Mathematics and Statistics,
	Central China Normal University, Wuhan, Hubei 430079, China}
\email{yliu@mails.ccnu.edu.cn}

\author{Jian Zhang}
\address{School of Mathematics and Statistics,
and Hubei Key Laboratory of Mathematical Sciences,
Central China Normal University, Wuhan, Hubei 430079, China}
\email{jzhang@ccnu.edu.cn}

\thanks{{\scriptsize
\hskip -0.6 true cm MSC (2020): Primary: 05E10; Secondary: 17B37, 58A17, 15A75, 15B33, 15A15.
}}


\begin{abstract}
Multiparametric quantum semigroups $\mathrm{M}_{\hat{q}, \hat{p}}(n)$ are generalization of the one-parameter general linear semigroups $\mathrm{M}_q(n)$, where $\hat{q}=(q_{ij})$ and $\hat{p}=(p_{ij})$ are $2n^2$ parameters satisfying certain conditions. In this paper, we study the algebra of multiparametric Manin matrices using
the R-matrix method. The systematic approach enables us to obtain
several classical identities such as Muir's identities, Newton's identities, Capelli-type identities, Cauchy-Binet's identity both for determinant and permanent
as well as a rigorous proof of the MacMahon master equation for the quantum algebra of multiparametric Manin matrices.
Some of the generalized identities are also lifted to multiparameter $q$-Yangians.
\end{abstract}
\maketitle

\section{Introduction}
 An $N\times N$ square matrix $M=(M_{ij})$ over an associative algebra $\mathcal A$ is called a {\it Manin matrix} if its entries satisfy the following relations:
\begin{equation}\label{e:Ma1}
M_{ik}M_{jl}-M_{jl}M_{ik}=M_{jk}M_{il}-M_{il}M_{jk}
\end{equation}
for all indices $i, j, k, l\in\{1, \ldots, N\}$. If $\mathcal A$ is commutative, then any square matrix satisfies \eqref{e:Ma1}.
Relations \eqref{e:Ma1} imply that all entries on the same column pairwise commute. The general entries of $M$
satisfy interesting algebraic relations in various classes of Manin matrices, which include the generator matrices of the Lie algebra $\mathfrak{gl}_N$, the affine Kac-Moody Lie algebra $\widehat{\mathfrak{gl}}_N$ and the
Yangian $Y({\mathfrak{gl}}_N)$ (cf. \cite{Mo}).

The matrix $M$ \eqref{e:Ma1} is in fact a ``$q=1$'' version of the $q$-Manin matrix introduced by Manin
in the study of quantum groups \cite{Ma1, Ma2}. He showed that the so-called quantum general linear semigroup can be viewed as the transformation group of a pair of quantum hyperplanes under comultiplication, i.e. the images of $y=Mx$ and $y=M^Tx$ both satisfy the quantum plane conditions: $x_jx_i=qx_ix_j, i<j$ if and only if $M$ is a quantum general linear semigroup matrix. Equivalently, it can also be viewed as the transformation group under comudule action preserving the quantum plane $x_jx_i=qx_ix_j, i<j$ and quantum exterior algebra $x_jx_i=-qx_ix_j, i<j$
and $x_i^2=0$.

We will consider an even more general version of Manin matrix $M=(M_{ij})$ whose noncommutative entries belong to
the quadratic algebra subject to the relations \cite{JZ2, S}:
\begin{align}
 &M_{ik}M_{jk}=q_{ji}M_{jk}M_{ik}, \quad 1\leq i<j\leq n, 1\leq k\leq m, \\
&M_{ik}M_{jl}-q_{ji}p_{kl}M_{jl}M_{ik}+ p_{kl}M_{il}M_{jk}- q_{ji}M_{jk}M_{il}=0.
\end{align}
which are exactly the transformation matrix for the comultiplication map obeying the quantum plane relations $x_jx_i=p_{ij}x_ix_j$ and $y_jy_i=q_{ij}y_iy_j$, i.e. they form a multiparametric quantum general linear semigroup.

To pass from quantum semigroups to quantum groups in Hopf algebras for one, two or multi-parameter cases \cite{RTF, BW, R}
is to add an antipode. In our general case of Manin matrices, we introduce the multi-parametric quantum column determinant \cite{Ma2}:

\begin{equation}
{\cdet}_{\hat{q}}(M)=\sum_{\sigma\in S_n}\varepsilon(\hat q, \sigma)M_{\sigma(1),1}\cdots M_{\sigma(n),n}
\end{equation}
where $$\varepsilon(\hat q, \sigma)=\prod_{i<j\atop \sigma(i)>\sigma(j)}(-q_{\sigma(i)\sigma(j)}).$$

The quantum general linear algebra offers a platform to generalize many classical identities. One of the first was the Cayley-Hamilton theorem  on the relation between determinant and linear operator, which was first generalized to one-parameter group \cite{JJZ} and later for multi-parameter general linear group. Another is the MacMahon master theorem \cite{Mac, EP}, one of the jewels of classical invariant theory. MacMahon's theorem was also first proved for special case of the quantum general linear group in a nontrivial way \cite{GLZ}.
In \cite{MR} a new method of R-matrices was offered for the more general context of quantum general linear supergroups.


In this paper, using the R-matrix method we revisit the theory of Manin matrices and offer systematic proofs of
 some important identities and also prove new ones involving with the quantum determinant and their generalizations. We will show rigorously in the general context that the quantum MacMahon master theorem also holds for the quantum multiparametric group. As new identities, for example, we show that the Muir theorem holds in the multiparametric case as well.

We also obtain Capelli-type identities and Cayley-Hamilton's theorem for multiparametric Manin matrices and
multparameter $q$-Yangians.

The paper is organized as follows. In Section \ref{s:manin} we study general properties of Manin matrices in the multiparametric case. Then we introduce the minor determinant of  the $(A_{\hat{q}},\ A_{\hat{p}})$-Manin matrix and derive their various properties: the Laplace expansion, the Pl\"ucker relations and the Cauchy-Binet identity. In Section \ref{s:gen-identities} we further generalize several well-known identities to Manin matrices. These include the well-known Jacobi's ratio identity, Cayley's complimentary identity, the Muir law,
and the Sylvester theorem. In Section \ref{s:gen-CauBin formulas} we study Capelli-type identities for determinants and permanents of  Manin matrices. In Section \ref{s:macmahon} the MacMahon master theorem is generalized to the general $A$-Manin matrices and our proof clarifies some of the missing points in the literature. In section \ref{s:cayley-hamilton}, we also generalize the Cayley-Hamilton theorem to the situation of Manin matrices. Finally in section \ref{s:multiYangian}, we introduce multiparametric Yangians and derive generalized versions of the Cayley-Hamilton theorem, Muir's identity and Newton's identity using Manin-matrices.

\section{Manin matrices}\label{s:manin}

Manin matrices dated back to Manin \cite{Ma1, Ma2}, see \cite{CFR, CFRS, S} for detailed accounts on Manin matrices, see also Molev's book \cite{Mo} for some recent applications in Yangians.
In the following we will focus on a generalized version of Manin matrices introduced by Silantyev \cite[Sect. 3.3]{S}.

Let $A\in \mathrm{End}(\mathbb{C}^{n}\otimes \mathbb{C}^{n})$
and $B\in \mathrm{End}(\mathbb{C}^{m}\otimes \mathbb{C}^{m})$
 be two idempotents and $\mathfrak{R}$ an algebra. We follow the convention that an endomorphism acts on column vectors from the left and row vectors from the right.
Define $\mathfrak{X}_{B}(\mathbb{C})$ and  $\Xi_{A}(\mathbb{C})$ as the (complex) quadratic algebras generated by the elements $x_{1}, \ldots, x_m$ and $\psi_{1}, \ldots, \psi_{n}$ over $\mathbb C$ respectively with the following  relations:
\begin{align}\label{e:relB}
&BX\otimes X=0,\\ \label{e:relA}
&(\Psi\otimes\Psi)(1-A)=0,
\end{align}
where $X=(x_1,\ldots,x_m)^t$ and $  \Psi=(\psi_1, \ldots, \psi_n).$

Extending the coefficients from $\mathbb C$ to $\mathfrak R$, we denote the algebra
$\mathfrak{R}\otimes \mathfrak{X}_{B}(\mathbb{C})$
($\mathfrak{R}\otimes \Xi_{A}(\mathbb{C})$ resp.)
by $ \mathfrak{X}_{B}(\mathfrak{R})$
( $\Xi_{A}(\mathfrak{R} )$ resp.).
Let $M$ be an $n\times m$ matrix with entries $M_{ij}\in \mathfrak{R}$.
 The space $\mathrm{Hom}(\mathfrak{R}^{m}, \mathfrak{R}^{n})\otimes \mathrm{Hom}(\mathfrak{R}^{m}, \mathfrak{R}^{n})$ is an $(\mathrm{End}(\mathfrak{R}^m), \mathrm{End}(\mathfrak{R}^n))$-bimodule with the canonical left and right action. For an element $M\in
\mathrm{Hom}(\mathfrak{R}^{m}, \mathfrak{R}^{n})$, we define
\begin{equation*}
M_1=M\otimes 1 , \qquad M_2=1\otimes M \in \mathrm{Hom}(\mathfrak{R}^{m}, \mathfrak{R}^{n})^{\otimes 2}.
\end{equation*}

Define $Y=MX$ and $\Phi=\Psi M$, i.e.
\begin{equation}\label{e:y-phi}
y_i=\sum_{j=1}^mM_{ij}x_j, \qquad \phi_j=\sum_{i=1}^n\psi_iM_{ij}.
\end{equation}
\begin{proposition}\label{p:comm}
The following  three conditions are equivalent:
\begin{align}\label{relation manin matrix}
&A M_{1}M_{2}(1-B)=0\\
&A(Y\otimes Y)=0\\
&(\Phi\otimes\Phi)(1-B)=0
\end{align}
\end{proposition}
An $n\times m$ matrix $M$ over an algebra $\mathfrak{R}$ satisfying the equation \eqref{relation manin matrix} is called an {\it $(A,\ B)$-Manin matrix}.

Let $P_A=1-2A$, $P_B=1-2B$, $S_A=1+P_A$, $S_B=1+P_B,$ then $P_A^2=P_B^2=1$.
Note that the quadratic  relation can be written as
\begin{align}
&P_AX\otimes X=X\otimes X,\\
&(\Psi\otimes\Psi)P_A=0.
\end{align}
The relation \eqref{relation manin matrix} is equivalent to
\begin{align}
&(1-P_A)M_1M_2(1+P_B)=0.
\end{align}

An $n\times n$ matrix $\hat{q}=(q_{ij})$ is called a  parametric matrix if its entries $q_{ij},1\leq i,j\leq n$ satisfying the conditions
\begin{equation}
q_{ij}q_{ji}=1,\qquad q_{ii}=1.
\end{equation}
So the transpose $\hat{q'}=(q_{ji})$ is also a parametric matrix.

Define $P_{\hat{q}}=\sum_{i,j=1}^n q_{ji}E_{ij}\otimes E_{ji}$.
Then $P_{\hat{q}}^2=1$ , $A_{\hat{q}}=\frac{1-P_{\hat{q}}}{2}$ and $S_{\hat{q}}=\frac{1+P_{\hat{q}}}{2}$ are idempotents.
The $m\times m$  parametric matrix $\hat{p}=(p_{ij})$, operator $P_{\hat{p}}$, idempotents $A_{\hat{p}}$ and $S_{\hat{p}}$ are  defined similarly.

Explicitly the quadratic relation \eqref{e:relB} of  $\mathfrak{X}_{A_{\hat{p}}}(\mathbb{C})$ is written as
\begin{equation}
x_{j}x_{i}=p_{ij}x_{i}x_{j},  \qquad 1\leq i,j\leq m.
\end{equation}
Similarly, the algebra $\Xi_{A_{\hat{q}}}(\mathbb{C})$ is defined by
\begin{equation}
 \psi_{i}^2=0, \quad  \psi_{j}\psi_{i}=-q_{ji}\psi_{i}\psi_{j}\  \text{for}  \quad 1\leq i\neq j\leq m.
\end{equation}

An {\it $(A_{\hat{q}},\ A_{\hat{p}})$-Manin matrix} $M$ is an $n\times m$ matrix over an algebra $\mathfrak{R}$ satisfying the following relation:
\begin{equation}\label{qp manin}
A_{\hat{q}}M_1M_2S_{\hat{p}}=0.
\end{equation}
In terms of entries this relation can be written as
  \begin{equation}\label{rel manin matrix}
  \begin{split}
  &M_{ik}M_{jk}=q_{ji}M_{jk}M_{ik}, \quad 1\leq i<j\leq n, 1\leq k\leq m, \\
&M_{ik}M_{jl}-q_{ji}p_{kl}M_{jl}M_{ik}+ p_{kl}M_{il}M_{jk}- q_{ji}M_{jk}M_{il}=0,
i<j,k<l.
\end{split}
\end{equation}
We will simply refer this $M$ as a $(\hat{q},\hat{p})$-{\it Manin matrix}.
 If  $n=m$ and  $\hat{q}=\hat{p}$, it is called a $(\hat{q})$-{\it Manin matrix}.

A multi-index is a tuple $I=(i_1,i_2,\ldots,i_r)$ of integers in $\{1,2,\ldots,n\}$.
The $\varepsilon$ symbol associated to the parameters $q_{ij}$ is defined by
\begin{equation}
\varepsilon(\hat q, I)
=
\left\{ \begin{aligned}
&0,\ &if\ two\ i's\ coinside,\\
&\prod_{s<t\atop i_s>i_t}(-q_{i_{s}{i_t}}) \ &if\ i's\ are\ distinct.\\
\end{aligned} \right.
\end{equation}
A multi-index $I=(i_1,i_2,\ldots,i_r)$ is called increasing multi-index if $i_1<i_2 <\cdots <i_r$, denoted by $I=(i_1<i_2<\ldots<i_r)$.

We denote by $I\oplus J=(i_1,\ldots,i_{r},j_1,\ldots,j_k)$ the juxtaposition 
of two multi-indices
$I=(i_1,\ldots,i_{r})$, $J=(j_1,\ldots,j_k)$.
Let $I=(i_1,\ldots,i_{r})$, $K=(k_1,\ldots,k_s)$ be multi-indices.
We say $I$ is contained in $K$ (still denoted as
$I\subset K$)
if there exist distinct $1', \ldots, r'\in\{1, \ldots, s\}$
such that $i_a=k_{a'}$ for $1\le a\le s$.

Let $I=(i_1,\ldots,i_{r})$ be a multi-index, $K=(k_1,\ldots,k_s)$ be another multi-index of increasing
integers and $I\subset K$.
Denote by $I^c$ 
the multi-index
$(k_1,\ldots,\hat{i_1},\ldots,\hat{i_r},\ldots,k_s)$ obtained from
 $(k_1,k_2,\ldots,k_s)$ by deleting $i_1,\ldots,i_r$, denoted $K\setminus I$.
 If $K=\{1,2,\ldots,n\}$ then $I^c=K \setminus I$ is increasing for any $I$.

Let $I=(i_1<i_2<\ldots<i_r)$ be a multi-index of increasing integers and $\sigma$
a permutation in $S_r$. The $\hat q$-inversion associated to the parameter $\hat q$ is defined as:
\begin{equation}
\varepsilon({\hat{q}},I,\sigma)=\prod_{\substack{s<t\\\si_s>\si_t}}
(-q_{i_{\sigma_{s}}i_{\sigma_t}}) .
\end{equation}
If $I=\{1,2,\ldots,n\}$, we simply denote it by $\varepsilon({\hat{q}}, \sigma)$.

Let $I=(i_1<i_2<\cdots<i_r)$ be a
multi-index of increasing positive integers and $J=(j_1,j_2,\cdots,j_r)$ be any  multi-index.
Denote by $M_{IJ}$  the matrix  whose row and column
indices belong to $I$ and $J$ respectively.
The  $\hat{q}$-minor column determinant is defined as (cf.\cite{JZ2})
\begin{equation}
\begin{split}
{\cdet}_{\hat{q}}(M_{IJ})
&=\sum_{\sigma\in S_r}  \varepsilon({\hat{q}},I,\sigma) M_{i_{\sigma(1)},j_1}\cdots M_{i_{\sigma(r)},j_r},
\end{split}
\end{equation}
In particular, for $m=n$
the column $\hat{q}$-determinant of $M$ is the column $n$-minor
\begin{equation}
\begin{split}
{\cdet}_{\hat{q}}(M)
&=\sum_{\sigma\in S_n} \varepsilon({\hat{q}}, \sigma)
M_{{\sigma(1)},1}\cdots M_{{\sigma(n)},n}.
\end{split}
\end{equation}

Using the relation $\psi_{j}\psi_{i}=-q_{ji}\psi_{i}\psi_{j} \, (i<j)$ and $\psi_i^2=0$, one has that
\begin{equation}
\phi_1 \phi_2 \cdots \phi_n
={\cdet}_{\hat{q}}(M)
\psi_1 \psi_2 \cdots \psi _n.
\end{equation}

\begin{proposition}
For any multi-index $I=(i_1,i_2,\ldots,i_n)$,
 one has for a square $(\hat{q}, \hat{p})$-Manin matrix $M$
\begin{equation*}
\sum_{\sigma\in S_n}
\varepsilon({\hat{q}}, \sigma)
M_{\sigma(1),i_1}\cdots M_{\sigma(n),i_n}=
\varepsilon(\hat p, I)
{\cdet}_{\hat{q}}(M).
\end{equation*}
\end{proposition}
\begin{proof}
The relation $(\Phi\otimes\Phi)S_{\hat{p}}=0$
can be written as
\begin{equation}
 \phi_{i}^2=0, \quad  \phi_{j}\phi_{i}=-p_{ji}\phi_{i}\phi_{j}\  \text{for}  \quad 1\leq i\neq j\leq m.
\end{equation}
Therefore,
\begin{align}\label{det1}
\phi_{i_1}\phi_{i_2}\cdots\phi_{i_n}=
\varepsilon(\hat p, I)\phi_{1}\phi_{2}\cdots\phi_{n}.
\end{align}
For any  $I=(i_1,i_2\cdots i_n)$ we can compute that
\begin{align}\label{det2}
&\phi_{i_1}\phi_{i_2}\cdots\phi_{i_n}
=\sum_{\sigma\in S_n}
\varepsilon({\hat{q}}, \sigma)
M_{\sigma(1),i_1}\cdots M_{\sigma(n),i_n}{\psi}_{1}{\psi}_{2}\cdots{\psi}_{n}.
\end{align}
Comparing equation \eqref{det1} with equation \eqref{det2}, we obtain the proposition.
\end{proof}

Let $I=(i_1,\ldots,i_r)$ and $K=(k_1,\ldots,k_{n-r})$ be two multi-indices.
Then
\begin{equation}\label{lap1}
\phi_{i_1}\cdots\phi_{i_r}
\phi_{k_{1}}\cdots\phi_{k_{n-r}}
=
\varepsilon({\hat{p}},I \oplus K)
{\cdet}_{\hat{q}}(M) {\psi}_{1}{\psi}_{2}\cdots{\psi}_{n}.
\end{equation}

On the other hand,
\begin{equation}\label{lap2}
\begin{split}
&\phi_{i_1}\cdots\phi_{i_r}
\phi_{k_{1}}\cdots\phi_{k_{n-r}}\\
=&\sum_{J}
{\cdet}_{\hat{q}}(M_{JI})
\psi_{j_1} \cdots \psi_{j_r}
{\cdet}_{\hat{q}}(M_{J^c, K} )
\psi_{j_{r+1}} \cdots\psi_{j_{n}}\\
=&\sum_{J}
 \varepsilon ({\hat{q}},J \oplus J^c)
{\cdet}_{\hat{q}}(M_{JI})
{\cdet}_{\hat{q}}(M_{ J^c, K} )
\psi_{1} \cdots\psi_{n},
\end{split}
\end{equation}
where the sum is over all multi-indices $J$ of increasing integers.
Comparing equations \eqref{lap1} and \eqref{lap2}, we obtain the following Laplace expansion.
\begin{proposition}[Laplace expansion]
Let $I$ and $K$ be two  subsets of $[1,n]$ with cardinality $r$ and $n-r$ respectively, then
\begin{equation}
\varepsilon({\hat{p}},I \oplus K)
{\cdet}_{\hat{q}}(M)
=\sum_{J}
 \varepsilon ({\hat{q}},J \oplus J^c)
{\cdet}_{\hat{q}}(M_{JI})
{\cdet}_{\hat{q}}(M_{ J^c, K} )
\end{equation}
where the sum is taken over all increasing multi-index $J$.
\end{proposition}

Taking $I=K$ in the Laplace expansion, we obtain the
Pl\"{u}cker\ relation, which generalizes the one-parameter one or the $q$-Maya relation (cf. \cite{JZ1}).
\begin{corollary}[Pl\"{u}cker\ relation]
Let $I=(i_1<\ldots<i_r)$ and $K=(k_1<\ldots<k_{2r})$ be two multi-indices of increasing integers, then
\begin{equation}
\sum_{J\subset K}
 \varepsilon ({\hat{q}},J\oplus K\setminus J)
{\cdet}_{\hat{q}}(M_{JI})
{\cdet}_{\hat{q}}(M_{K\setminus J, I})
=0,
\end{equation}
where the sum is taken over all multi-indices $J\subset K$ of increasing positive
integers.
\end{corollary}

The  special cases ($r=n-1$) of the Laplace expansion reads:
\begin{equation}\label{cramer}
\sum_{j}
 \varepsilon({\hat{q}},j^c\oplus j)
{\cdet}_{\hat{q}}(M_{j^c,i^c})
{\cdet}_{\hat{q}}(M_{ j, k})
= \varepsilon({\hat{p}},i^c\oplus k){\cdet}_{\hat{q}}(M).
\end{equation}

The following proposition follows from \eqref{cramer}.
\begin{proposition}[Inverse of Manin matrix]\label{inverse of manin}
Let $M$ be an
$n\times n$ $(\hat{q}, \hat{p})$-Manin matrix such that ${\cdet}_{\hat{q}}(M)$ is left invertible.
Then $M$ is left invertible and the left inverse $M^{-1}$ of $M$ is given by
\begin{equation}
(M^{-1})_{ij}=
 \varepsilon({\hat{p}},i^c\oplus i)^{-1}
 \varepsilon({\hat{q}},j^c\oplus j)
{\cdet}_{\hat{q}}(M)^{-1}
{\cdet}_{\hat{q}}(M_{j^c i^c}),
 \end{equation}
where  ${\cdet}_{\hat{q}}(M)^{-1}$ is the left inverse of ${\cdet}_{\hat{q}}(M)$.
\end{proposition}

Moreover, if $M$ is also right invertible then the left inverse and the right inverse coincide.
We will call an element or matrix invertible if it is left and right invertible.

In the following we assume that any square submatrix of a Manin matrix is invertible and the
column minor determinant of the submatrix is left invertible.

Suppose that $X$ is invertible with $X^{-1}=Y$,  and $y_{ji}$ is invertible. The $(i,j)$-th quasideterminant $|X|_{ij}$
is the following element \cite{GR, GR2, GGRW}:
$$|X|_{ij}={(y_{ji})}^{-1}.$$

For any subset $I\subset \{1,2,\ldots,n\}$, let $X_{I}$ denote the submatrix whose row and column indices belong to $I$. For any multi-index $I=(i_1,i_2,\ldots,i_r)$, there exists an permutation $\sigma\in S_r$ such that
$i_{\sigma_1}\leq i_{\sigma_2}\leq \cdots\leq i_{\sigma_r}$.
We denote the ordered multi-index
$(i_{\sigma_1}\leq i_{\sigma_2}\leq \cdots\leq i_{\sigma_r})$
by $I^{or}$.

The quantum determinant of one-parameter general linear semigroups $\mathrm{M}_q(n)$ can be expressed as a product of minor quasideterminants \cite{GR2, KL}.
The following theorem is a generalization of this result for multiparametric Manin matrices.
\begin{theorem}
Let $M$ be an $n\times n$ $(\hat{q},\hat{p})$-Manin matrix, and let
$I=(i_1,i_2,\ldots,i_n)$ and  $J=(j_1,j_2,\ldots,j_n)$
be two permutations of $\{1,2,\ldots,n\}$.
Let $I_k=\{i_1,\ldots,i_k\}^{or}$ and $J_k=\{j_1,\ldots,j_k\}^{or}$ be the ordered sub-multi-indices, then
the determinant can be expressed as a product of quasideterminants
\begin{equation}
\begin{split}
{\cdet}_{\hat{q}}(M)
=
\frac{ \varepsilon({\hat{q}},J)}{ \varepsilon({\hat{p}},I)}
|M_{J_1I_1}|_{j_1i_1}|M_{J_2I_2}|_{j_2i_2}\cdots|M_{J_{n-1}I_{n-1}}|_{j_{n}i_{n}}
\end{split}
 \end{equation}
\end{theorem}

\begin{proof} We use induction on $n$. First of all
it follows from Proposition \ref{inverse of manin} that
\begin{equation}
\begin{split}
{\cdet}_{\hat{q}}(M)&=
\frac{ \varepsilon({\hat{q}},j^c_n\oplus j_n)}{ \varepsilon({\hat{p}},i^c_n\oplus i_n)}
{\cdet}_{\hat{q}}(M_{j^c_ni^c_n})
(M^{-1})_{i_nj_n}^{-1}\\
&=
\frac{ \varepsilon({\hat{q}},j^c_n\oplus j_n)}{ \varepsilon({\hat{p}},i^c_n\oplus i_n)}
{\cdet}_{\hat{q}}(M_{j^c_ni^c_n})
|M|_{j_ni_n}
\end{split}
 \end{equation}

By induction hypothesis we have that
\begin{equation}
\begin{split}
{\cdet}_{\hat{q}}(M_{j^c_ni^c_n})
=
\frac{ \varepsilon({\hat{q}},J_{n-1})}{ \varepsilon({\hat{p}},I_{n-1})}
|M_{J_1I_1}|_{j_1i_1}\cdots|M_{J_{n-1}I_{n-1}}|_{j_{n-1}i_{n-1}}.
\end{split}
 \end{equation}

Therefore,
\begin{equation}
\begin{split}
{\cdet}_{\hat{q}}(M)
=
\frac{ \varepsilon({\hat{q}},J)}{ \varepsilon({\hat{p}},I)}
|M_{J_1I_1}|_{j_1i_1}|M_{J_2I_2}|_{j_2i_2}\cdots|M_{J_{n-1}I_{n-1}}|_{j_{n}i_{n}}.
\end{split}
 \end{equation}
\end{proof}

The following is an analogue of Cauchy-Binet's formula for multiparametric Manin matrices.
\begin{proposition}[Cauchy-Binet's formula]\label{det cauchy-Binet}
Let $M$ be an $n\times m$ $(\hat{q},\hat{p})$-Manin matrix  and $N$ an $m\times s$ matrix such that the $N_{ij}$  commute with $M_{kl}$ and $\psi_t$ for all possible indices  $i,j,k,l,t$.
Let $I=(i_1<\cdots<i_r)$ and $K=(k_1,\cdots,k_r)$ be two multi-indices.
Then
 \begin{equation}
{\cdet}_{\hat{q}}((MN)_{IK})=\sum_{J} {\cdet}_{\hat{q}}((M)_{IJ}){\cdet}_{\hat{p}}((N)_{JK}),
  \end{equation}
for $r\leq m$, where the sum is taken over all multi-indices of increasing integers $J\subset(1,\ldots,m)$.
And
 \begin{equation}
{\cdet}_{\hat{q}}((MN)_{IK}=0,
\end{equation}
for $r>m$.
In particular, if  $m=n=s$, then
 \begin{equation}
{\cdet}_{\hat{q}} (MN) ={\cdet}_{\hat{q}}(M){\cdet}_{\hat{p}}(N).
  \end{equation}
\end{proposition}
\begin{proof}
Let $\xi_{i}=\sum_{j=1}^n\psi_j (MN)_{ji}$. It follows from \eqref{e:y-phi} that
\begin{equation}
\xi_{i}=\sum_{j=1}^n\psi_j(MN)_{ji}
=\sum_{j=1}^n\sum_{k=1}^m\psi_jM_{jk}N_{ki}
=\sum_{k=1}^m\phi_kN_{ki}.
\end{equation}
Then we have that
\begin{equation}\label{CB1}
\begin{split}
\xi_{k_1}\xi_{k_2}\cdots\xi_{k_r}
&=
\sum_{i_1,\ldots i_r=1}^n \psi_{i_1}\cdots\psi_{i_r}(MN)_{i_1k_1}\cdots (MN)_{i_rk_r},
\end{split}
\end{equation}
which is zero if $r>m$. For $r\leq m$, one has
\begin{equation}
\begin{split}\xi_{k_1}\xi_{k_2}\cdots\xi_{k_r}
&=\sum_{1\leq i_1<\ldots <i_r\leq n} \psi_{i_1}\cdots\psi_{i_r}{\cdet}_{\hat{q}}((MN)_{IK}),
\end{split}
\end{equation}
where $I=(i_1<\cdots<i_r)$ is a multi-index of increasing integers.

On the other hand,
\begin{equation}\label{CB2}
\begin{split}
\xi_{k_1}\xi_{k_2}\cdots\xi_{k_r}
&=\sum_{j_1,\ldots,j_r=1}^m\phi_{j_1}\cdots \phi_{j_r}  N_{j_1k_1}\cdots N_{j_r k_r}\\
&=\sum_{1\leq j_1<\ldots <j_r\leq m}\phi_{j_1}\cdots \phi_{j_r} {\cdet}_{\hat{p}}(N_{JK})\\
&=\sum_{I,J} \psi_{i_1}\cdots\psi_{i_r} {\cdet}_{\hat{q}}((M)_{IJ}) {\cdet}_{\hat{p}}(N_{JK})\\
\end{split}
\end{equation}
where the sum is taken over all multi-indices of increasing integers
$I=(i_1<\ldots<i_r)$ and $J=(j_1<\ldots<j_r)$.
Comparing the coefficients of $\psi_{i_1}\cdots\psi_{i_r}$, one has the proposition.
\end{proof}

For any multi-index $I=(i_1,i_2,\ldots,i_r)$, we denote the reverse
of $I$
by $$I^{\tau}=(i_r,i_{r-1},\ldots,i_1).$$
In particular, we denote the reverse of $(1,2,\ldots,n)$ by $\tau$. Recall that $I^{or}$ is the ordered
multi-index of $I$.

\begin{proposition}\label{right invertible}
Let $M$ be a right invertible $n\times n$ $(\hat{q},\hat{p})$-Manin matrix,
$I$ and $J$ be subsets of $[1,n]$ of cardinality $m\leq n$.
Then
\begin{equation}
\sum_{K}
 \varepsilon ({\hat{p}},K^{\tau})
 {\cdet}_{\hat{q}}(M _{IK})
{\cdet}_{\hat{p'}}(M^{-1}_{KJ})
=\varepsilon ({\hat{q}},J^{\tau})\delta_{IJ^{or}}
\end{equation}
where
 the sum is taken over all increasing multi-index $K\subset (1,\ldots,n)$,
 $p'_{ij}=p_{ij}^{-1}$.
In particular, for $m=n$  one has that ${\cdet}_{\hat{q}}(M )$ 
is right invertible and
\begin{equation}
 {\cdet}_{\hat{q}}(M )
{\cdet}_{\hat{p'}}(M^{-1})
=
 \varepsilon ({\hat{p}},\tau)^{-1}
 \varepsilon ({\hat{q}},\tau)
\end{equation}
\end{proposition}
\begin{proof}
We write the element ${\psi}_i$ as $\sum_{j=1}^n \phi_j  M^{-1}_{ji}$ (see \eqref{e:y-phi}).
The elements $M^{-1}_{ij}$ commute with
${\psi}_k$ for all possible $i,j,k$. Then
\begin{equation}
\begin{split}
{\psi}_{j_m} \cdots {\psi}_{j_1}
&=
\sum_{k_m=1}^{n} \phi_{k_m}  M^{-1}_{k_m j_m}{\psi}_{j_{m-1}} \cdots {\psi}_{j_1}\\
&=
\sum_{k_m=1}^n \phi_{k_m}  {\psi}_{j_{m-1}} \cdots {\psi}_{j_1} M^{-1}_{k_m j_m}\\
&=
\sum_{k_1,\ldots,k_m=1}^n \phi_{k_m} \phi_{k_{m-1}}\cdots \phi_{k_{1}}
M^{-1}_{k_1 j_1}  M^{-1}_{k_2 j_2} \cdots
M^{-1}_{k_m j_m}\\
\end{split}
\end{equation}
Using the relations $\phi_{j}\phi_{i}=-p_{ji}\phi_{i}\phi_{j}$,
we have that
\begin{equation}
\begin{split}
{\psi}_{j_m} &\cdots {\psi}_{j_1}
=
\sum_{1\leq k_1<\ldots<k_m\leq n}
 \varepsilon ({\hat{p}},K^{\tau})
\phi_{k_1} \cdots \phi_{k_{m}}
{\cdet}_{\hat{p'}}(M^{-1}_{KJ})\\
&=
\sum_{I,K }
 \varepsilon ({\hat{p}},K^{\tau})
 {\cdet}_{\hat{q}}(M _{IK})
{\cdet}_{\hat{p'}}(M^{-1}_{KJ}){\psi}_{i_1} \cdots {\psi}_{i_m}
\end{split}
\end{equation}
where the sum is taken over all possible multi-indices $I=(i_1<\ldots<i_m)$, $K=(k_1<\ldots<k_m)$.

Comparing the coefficients of ${\psi}_{i_m} \cdots {\psi}_{i_1}$,
we obtain the proposition.
\end{proof}

\section{Further generalization of classical identities}\label{s:gen-identities}
The map
$s_i=(i, i+1)\mapsto {P_{\hat{q}}}^{s_i}$
defines an $S_k$-module structure on ${\mathbb{C}^{n}}^{\otimes k}$, where $1\leq i\leq k-1$.
If $\sigma=\sigma_{i_1}\cdots\sigma_{i_l}$, we set ${P_{\hat{q}}}^{\sigma}={P_{\hat{q}}}^{\sigma_{i_1}}\cdots {P_{\hat{q}}}^{\sigma_{i_l}}$, which is well-defined and independent from the choice of reduced expression.
By $S_{\hat{q}}^{(k)}$ and $A_{\hat{q}}^{(k)}$ we denote the respective images of the normalized symmetrizer and antisymmetrizer:
\begin{equation}
S_{\hat{q}}^{(k)}=\frac{1}{k!}\sum_{\sigma\in S_k} {P_{\hat{q}}}^{\sigma}, \qquad
A_{\hat{q}}^{(k)}=\frac{1}{k!}\sum_{\sigma\in S_k}\text{sgn}\ \sigma {P_{\hat{q}}}^{\sigma}.
\end{equation}
In particular,
\begin{equation}
S_{\hat{q}}^{(2)}=S_{\hat{q}} , \qquad
A_{\hat{q}}^{(2)}=A_{\hat{q}}.
\end{equation}

\begin{proposition}\label{p:comm2}
One has the following identities in $\mathfrak{R} \otimes {\mathbb{C}^{n}}^{\otimes k}$
\begin{align}
A_{\hat{q}}^{(k)}M_1\cdots M_k=A_{\hat{q}}^{(k)} M_1\cdots M_k A_{\hat{p}}^{(k)},\\
M_1\cdots M_k S_{\hat{p}}^{(k)}=S_{\hat{q}}^{(k)}M_1\cdots M_k S_{\hat{p}}^{(k)}.
\end{align}
\end{proposition}

\begin{proof} This follows from Proposition \ref{p:comm}.\end{proof}

Define the operator $A_{\hat{q}\hat{p}}^{(k)}$ in
$\mathrm{End}({\mathbb{C}^{n}}^{\otimes k})$ by

\begin{align}
A_{\hat{q}\hat{p}}^{(k)}=\frac{1}{k!}
\sum_{\sigma,\rho \in S_k
\atop 1\leq i_1<\cdots<i_k\leq n}
\frac{\varepsilon(\hat p, I,\rho)}{\varepsilon(\hat q,I,\sigma)}
e_{i_{\sigma(1)}i_{\rho(1)}}\ot \cdots \ot e_{i_{\sigma(k)}i_{\rho(k)}},
\end{align}
where $I=(i_1<\cdots<i_k)$. In particular, for $k=n$
\begin{align}
A_{\hat{q}\hat{p}}^{(n)}=\frac{1}{n!}
\sum_{\sigma,\rho \in S_n}
\frac{\varepsilon(\hat p,\rho)}{\varepsilon(\hat q,\sigma)}
e_{{\sigma(1)}{\rho(1)}}\ot \cdots \ot e_{{\sigma(n)}{\rho(n)}}.
\end{align}

The following result is an easy consequence of Prop. \ref{p:comm2}.
\begin{proposition} The following identity holds
in $\mathfrak{R} \otimes {\mathbb{C}^{n}}^{\otimes k}$:
\begin{align}
A_{\hat{q}}^{(k)}M_1\cdots M_k
=
{\cdet}_{\hat{q}}(M)
A_{\hat{q}\hat{p}}^{(k)}.
\end{align}
\end{proposition}

The  following theorem is an analog of Jacobi's ratio theorem  for multiparametric Manin matrices.

\begin{theorem}[Jacobi's ratio Theorem]\label{Jacobi thm}
Let $M$ is an invertible $n\times n$ $(\hat{q},\hat{p})$-Manin matrix,  and let
 $I=(i_1<i_2<\cdots<i_k)$ be a multi-index of increasing integers and $J=(j_1,j_2,\cdots,j_k)$ be any multi-index.
Then
\begin{align}
\varepsilon (\hat {p},I^c\oplus I^{\tau})  {\cdet}_{\hat{q}}(M)
{\cdet}_{\hat{p'}}(M^{-1}_{IJ})
=
\varepsilon (\hat {q},J^c\oplus J^{\tau}) {\cdet}_{\hat{q}}(M_{J^cI^c}).
\end{align}
\end{theorem}

\begin{proof}
Multiplying the equation
\begin{equation}
A_{\hat{q}}^{(n)}M_1\cdots M_n
=
{\cdet}_{\hat{q}}(M)
A_{\hat{q}\hat{p}}^{(n)}.
\end{equation}
 by $M_n^{-1}\cdots M_{n-k+1}^{-1}$ from the right, one get that
\begin{equation}\label{e:AM}
A_{\hat{q}}^{(n)}M_1\cdots M_{n-k}
=
{\cdet}_{\hat{q}}(M)
A_{\hat{q}\hat{p}}^{(n)}M_n^{-1}\cdots M_{n-k+1}^{-1}.
\end{equation}

Let $I^c=(i_{k+1}<\cdots<i_{n})$.
Applying both sides of \eqref{e:AM} to the vector $v=e_{i_{k+1}}\otimes\dots \otimes e_{i_n} \otimes e_{j_{k}} \otimes\dots \otimes e_{j_{1}}$.
we have that
\begin{equation}\label{co R}
\begin{split}
&{\cdet}_{\hat{q}}(M)
A_{\hat{q}\hat{p}}^{(n)}
M_n^{-1}\cdots M_{n-k+1}^{-1}v\\
=&{\cdet}_{\hat{q}}(M)
{\cdet}_{\hat{p'}}(M^{-1}_{IJ})
A_{\hat{q}\hat{p}}^{(n)}
e_{i_{k+1}}\otimes\dots \otimes e_{i_n} \otimes e_{i_{k}} \otimes\dots \otimes e_{i_{1}},
\end{split}
\end{equation}
where the coefficient of $e_{1}\otimes\dots \otimes e_{n}$ 
is
$\varepsilon (\hat {p},I^c\oplus I^{\tau})  {\cdet}_{\hat{q}}(M)
{\cdet}_{\hat{p'}}(M^{-1}_{IJ})$.

If $J$ has two equal indices, then $A_{\hat{q}}^{(n)}M_1\cdots M_{n-k}v=0$.
If $J$ does not have two equal indices, say $J^c=(j_{k+1}<\cdots<j_{n})$, then
\begin{equation}\label{co L}
\begin{split}
&A_{\hat{q}}^{(n)}M_1\cdots M_{n-k}
v\\
=&
A_{\hat{q}}^{(n)}{\cdet}_{\hat{q}}(M_{J^cI^c})
e_{j_{k+1}}\otimes\dots \otimes e_{j_n} \otimes e_{j_{k}} \otimes\dots \otimes e_{j_{1}}
\end{split}
\end{equation}
The coefficient of $e_{1}\otimes\dots \otimes e_{n}$ in \eqref{co L} is
$\varepsilon (\hat {q},J^c\oplus J^{\tau}) {\cdet}_{\hat{q}}(M_{J^cI^c})$.
Note that $\varepsilon (\hat {q},J^c\oplus J^{\tau})=0$ if  $J$ has two equal indices.
This completes the proof.


\end{proof}

\begin{proposition}\label{manin matrix invertible}
If  $M$ is an invertible  $n\times n$ $(\hat{q},\hat{p})$-Manin matrix such that
${\cdet}_{\hat{q}}(M)$ is invertible, then
$M^{-1}$ is an $n\times n$ $(\hat{p'},\hat{q'})$-Manin matrix, where $\hat{q'}=(q_{ji})$ is the transpose of $\hat{q}$.
\end{proposition}
\begin{proof}
Let $I=(i<j)$, $J=(k,k)$.
By Theorem \ref{Jacobi thm}, we have
\begin{align}
{\cdet}_{\hat{p'}}(M^{-1}_{IJ})
=
M^{-1}_{ik}M^{-1}_{jk}-p_{ij}M^{-1}_{jk}M^{-1}_{ik}=0.
\end{align}

Let $I=(i<j)$, $J=(k<l)$.
Then
\begin{align}
\varepsilon (\hat {q},J^c\oplus J^{\tau})
=(-q_{kl})^{-1}\varepsilon (\hat {q},J^c\oplus J).
\end{align}
It follows from  Theorem \ref{Jacobi thm} that
\begin{align}
{\cdet}_{\hat{p'}}(M^{-1}_{IJ})=(-q_{kl})^{-1}{\cdet}_{\hat{p'}}(M^{-1}_{IJ^{\tau}}).
\end{align}
Therefore,
\begin{align}
M^{-1}_{ik} M^{-1}_{jl}- p_{ij}M^{-1}_{jk}M^{-1}_{il}
=(-q_{kl})^{-1}(M^{-1}_{il} M^{-1}_{jk}- p_{ij}M^{-1}_{jl}M^{-1}_{ik}).
\end{align}
So
$M^{-1}$ is a $(\hat{p'},\hat{q'})$-Manin matrix.
\end{proof}

From the generalized Jacobi's ratio Theorem \ref{Jacobi thm} we have the following analogs of Cayley's complementary identity, Muir's law and Sylvester's theorem for multiparametric Manin matrices.

\begin{theorem}[Cayley's complementary identity]\label{cayley thm }
Let $M$ be an invertible $n\times n$ $(\hat{q},\hat{p})$-Manin matrix
and let ${\cdet}_{\hat{q}}(M)$ be invertible.
Suppose one is given a minor identity of determinants
\begin{equation}
\sum_{r=1}^{k}b_r \prod_{s=1}^{m_r}
{\cdet}_{\hat q} (M_{I_{rs}J_{rs}})=0,
\end{equation}
where  $I_{rs}$ are increasing multi-indices, $J_{rs}$ are multi-indices of distinct integers and
 $b_r\in \mathbb C(q_{ij},p_{ij})$.
Then the following identity holds

\begin{equation}
\sum_{r=1}^{k} b_r'\prod_{s=1}^{m_r}
\frac{\varepsilon (\hat {q},J^c_{rs}\oplus J_{rs}^{\tau})}
{\varepsilon (\hat {p},I^c_{rs}\oplus I_{rs}^{\tau})} {\cdet}_{\hat{q}}(M)^{-1}{\cdet}_{\hat{q}}(M_{J^c_{rs}I^c_{rs}})=0,
\end{equation}
where $b_r'$ is obtained from $b_r$ by replacing $q_{ij}$ by $p_{ij}^{-1}$ and $p_{ij}$ by $q_{ij}^{-1}$ respectively.
\end{theorem}

\begin{proof}
Applying the minor identity to $M^{-1}$ we get that
\begin{equation}
\sum_{r=1}^{k}b_r' \prod_{s=1}^{m_r}
{\cdet}_{\hat {p'}} (M^{-1}_{I_{rs}J_{rs}})=0.
\end{equation}
Substituting
$\frac{\varepsilon (\hat {q},J^c_{rs}\oplus J_{rs}^{\tau})}
{\varepsilon (\hat {p},I^c_{rs}\oplus I_{rs}^{\tau})} {\cdet}_{\hat{q}}(M)^{-1}{\cdet}_{\hat{q}}(M_{J^c_{rs}I^c_{rs}})$
for  ${\cdet}_{\hat{p'}}(M^{-1}_{I_{rs}J_{rs}})$,
we get the result. 
\end{proof}

\begin{theorem}[Muir's law]\label{muir law}
Let $M$ be an  $n\times n$ $(\hat{q},\hat{p})$-Manin matrix and $K=\{1,2,\ldots,m\}$, $L=\{m+1,m+2,\ldots,n\}$, where $m<n$. Suppose that $M$, $M_{KK}$, $M_{LL}$, ${\cdet}_{\hat{q}}(M)$, ${\cdet}_{\hat{q}}(M_{KK})$ and ${\cdet}_{\hat{q}}(M_{LL})$ are invertible and the following minor identity holds
\begin{equation}
\sum_{r=1}^{k}b_r \prod_{s=1}^{m_r}
{\cdet}_{\hat q} (M_{I_{rs}J_{rs}})=0,
\end{equation}
where $I_{rs}$ are increasing multi-indices, $J_{rs}$ are multi-indices without equal indices and
$I_{rs},J_{rs}\subset L$, $b_r\in \mathbb C(q_{ij},p_{ij})$.
Then the following identity holds
\begin{equation}
\sum_{r=1}^{k} b_r\prod_{s=1}^{m_r}
{\cdet}_{\hat{q}}(M_{KK})^{-1}
 {\cdet}_{\hat{q}}(M_{  (K\cup I_{rs})(K\cup J_{rs}) })=0
\end{equation}
\end{theorem}
\begin{proof}
Applying Cayley's complementary identity respect to the set $L$, we get that
\begin{equation}\label{muir id}
\sum_{r=1}^{k} b_r'\prod_{s=1}^{m_r}
\frac{\varepsilon (\hat {q},L\setminus J _{rs}\oplus J_{rs}^{\tau})}
{\varepsilon (\hat {p},L \setminus  I_{rs}\oplus I_{rs}^{\tau})}
 {\cdet}_{\hat{q}}(M_{LL})^{-1}
 {\cdet}_{\hat{q}}(M_{L\setminus   J_{rs},L\setminus  I_{rs} })=0,
\end{equation}

Applying Cayley's
complementary identity respect to the set $\{1,\ldots,n\}$,
Equation \eqref{muir id} can be written as
\begin{equation}\label{Muir eq}
\begin{split}
\sum_{r=1}^{k} b_r\prod_{s=1}^{m_r}
f_{rs}g_{rs}h_{rs}=0
 \end{split}
\end{equation}
where
\begin{align*}
&f_{rs}=\frac{\varepsilon (\hat {p'},L\setminus J _{rs}\oplus J_{rs}^{\tau})}
{\varepsilon (\hat {q'},L \setminus  I_{rs}\oplus I_{rs}^{\tau})},\\
&g_{rs}=\frac{\varepsilon (\hat {p},K\oplus L^{\tau})}
{\varepsilon (\hat {q},K\oplus L^{\tau})}
{\cdet}_{\hat{q}}(M_{KK})^{-1}
{\cdet}_{\hat{q}}(M),\\
&h_{rs}=\frac{\varepsilon (\hat {q},(K\cup I_{rs})\oplus (L \setminus I _{rs})^{\tau})}
{\varepsilon (\hat {p},(K\cup J_{rs})\oplus (L\setminus J _{rs})^{\tau})}
{\cdet}_{\hat{q}}(M)^{-1}
 {\cdet}_{\hat{q}}(M_{  (K\cup I_{rs})(K\cup J_{rs}) }).
\end{align*}

For any increasing multi-indices $I$ and $K$, we have

\begin{equation}
\varepsilon (\hat {q}, I\oplus K^{\tau})=
\varepsilon (\hat {q}, I\oplus K )\varepsilon (\hat {q}, K^{\tau} ).
\end{equation}
Moreover, if $I\subset K$, then
\begin{equation}
\varepsilon (\hat {q}, K^{\tau})=
\varepsilon (\hat {q}, I^{\tau})\varepsilon (\hat {q}, (K\setminus I )^{\tau})
\varepsilon (\hat {q},I\oplus K\setminus I )
\varepsilon (\hat {q},  K\setminus I \oplus I).
\end{equation}

Therefore,
\begin{equation*}
\begin{split}
&\frac{\varepsilon (\hat {p'},L\setminus J _{rs}\oplus J_{rs}^{\tau})}
{\varepsilon (\hat {q'},L \setminus  I_{rs}\oplus I_{rs}^{\tau})}
\frac{\varepsilon (\hat {p},K\oplus L^{\tau})}
{\varepsilon (\hat {q},K\oplus L^{\tau})}
\frac{\varepsilon (\hat {q},(K\oplus I_{rs})\oplus (L \setminus I _{rs})^{\tau})}
{\varepsilon (\hat {p},(K\oplus J_{rs})\oplus (L \setminus J _{rs})^{\tau})}\\
&=\frac
{\varepsilon (\hat {q},L \setminus  I_{rs}\oplus I_{rs}^{\tau})}
{\varepsilon (\hat {p},L\setminus J _{rs}\oplus J_{rs}^{\tau})}
\frac{\varepsilon (\hat {p},  L^{\tau})}
{\varepsilon (\hat {q}, L^{\tau})}
\frac{\varepsilon (\hat {q},(  I_{rs})\oplus (L \setminus I_{rs})^{\tau})}
{\varepsilon (\hat {p},(  J_{rs})\oplus (L \setminus J _{rs})^{\tau})}\\
&=1,
 \end{split}
\end{equation*}
and
\begin{align*}
f_{rs}g_{rs}h_{rs}=
{\cdet}_{\hat{q}}(M_{KK})^{-1}
 {\cdet}_{\hat{q}}(M_{  (K\cup I_{rs})(K\cup J_{rs}) }).
\end{align*}
 This completes the proof.
\end{proof}

\begin{theorem}[Sylvester's Theorem]
Let $M$ be an  $n\times n$ $(\hat{q},\hat{p})$-Manin matrix and $K=\{1,2,\ldots,m\}$, $L=\{m+1,m+2,\ldots,n\}$,where $m<n$. Assume that $M$, $M_{KK}$, $M_{LL}$,  ${\cdet}_{\hat{q}}(M)$, ${\cdet}_{\hat{q}}(M_{KK})$ and ${\cdet}_{\hat{q}}(M_{LL})$ are invertible.
Let $B$ be a $m\times m$ matrix with
$$B_{rs}={\cdet}_{\hat{q}}(M_{LL})^{-1}
 {\cdet}_{\hat{q}}(M_{  (i_{r}\oplus L)(  i_{s} \oplus L)}).$$
  Then $B$ is a  $(\hat{q},\hat{p})$-Manin matrix.
Moreover,
\begin{equation}
{\cdet}_{\hat{q}}(B)={\cdet}_{\hat{q}}(M_{LL})^{-1}{\cdet}_{\hat{q}}(M).
\end{equation}
\end{theorem}

\begin{proof}

It follows from Muir's Law that $B$ is a  $(\hat{q},\hat{p})$-Manin matrix.
Applying Muir's law to the equation
\begin{equation}
\begin{split}
{\cdet}_{\hat{q}}(M_{KK})
&=\sum_{\sigma\in S_m} \varepsilon({\hat{q}}, \sigma)
M_{{\sigma(1)},1}\cdots M_{{\sigma(m)},m},
\end{split}
\end{equation}
we get that
\begin{equation}
{\cdet}_{\hat{q}}(B)={\cdet}_{\hat{q}}(M_{LL})^{-1}{\cdet}_{\hat{q}}(M).
\end{equation}
\end{proof}

\section{Capelli-type identities }\label{s:gen-CauBin formulas}
The Capelli identity is a celebrated result in the classical invariant theory
that provides a set of generators for the center of enveloping algebra $U(\mathfrak{gl}(n))$.
Let $E=(E_{ij})$, where $E_{ij} (1\leq i,j\leq n)$ are the basis elements of $\mathfrak{gl}(n)$, then
\begin{equation*}
\begin{aligned}
&\cdet(E+diag(n-1,\ldots,0))
\\
&=\sum_{\sigma\in S_n}sgn(\sigma)
(E_{\sigma(1)1}+(n-1)\delta_{\sigma(1)1})(E_{\sigma(2)2}+(n-2)\delta_{\sigma(2)2})\cdots E_{\sigma(n)n}
\end{aligned}
\end{equation*}
is a central element of $U(\mathfrak{gl}(n))$.
Let $\mathcal{PD}(\mathbb{C}^{n\times n})$ be the algebra generated by $x_{ij}$ and $\partial_{ij}$, $1\leq i,j\leq n$.
Denote $X=(x_{ij})$ and $D=(\partial_{ij})$, then
\begin{equation*}
E \mapsto X D^t
\end{equation*}
defines an algebra homomorphism from $U(\mathfrak{gl}(n))$ to $\mathcal{PD}(\mathbb{C}^{n\times n})$.
The image of the central element was given by Capelli \cite{C} as follows.
\begin{equation*}
\det(X D^t+diag(n-1,\ldots,0))=\det X\det D.
\end{equation*}

There have been many generalizations of the Capelli identities.
Turnbull \cite{Turnbull} considered the Capelli identity for determinant of symmetric matrices and permanent of antisymmetric matrices.
Howe and Umeda \cite{HU}, Kostant and Sahi \cite{KS} proved the Capelli identities for determinants in the antisymmetric case.
The generalization of Capelli identities to immanants was discovered by Willamson \cite{Willamson} and Okounkov \cite{Ok}.
In \cite{CSS,CFR}, the Capelli identities were generalized to Manin matrices.  The quantum analogue of the Capelli identity was given in \cite{NUW} and was
further generalized to the multiparametric quantum case in \cite{JZ2}.

In this section we study the Capelli-type identities for determinants and permanents of $(\hat q, \hat p)$-$Manin$ matrices such that all entries of $\hat q$ or $\hat p$ equal to 1.
In this case, we call the matrix  $(1, \hat p)$-$Manin$ matrix and  $(\hat q, 1)$-$Manin$ matrix respectively.
Define  $$P_{s \times m}=\sum_{i=1}^m \sum_{j=1}^s E_{ij}\otimes E_{ji} \in \mathrm{Hom}(\mathbb{C}^{s}\otimes \mathbb{C}^{m}, \mathbb{C}^{m}\otimes \mathbb{C}^{s}).$$
 Then $P_{m \times s} P_{s \times m}=1 \in \mathrm{End}(\mathbb{C}^{s}\otimes \mathbb{C}^{m})$.
For $1\leq a<b\leq r$, we denote
\begin{equation}
 P_{s \times m}^{(a,b)}=\sum_{i=1}^m \sum_{j=1}^s 1^{\otimes (a-1)}
\otimes E_{ij} \otimes 1^{\otimes (b-a-1)}\otimes E_{ji}\otimes 1^{\otimes (r-b)}.
\end{equation}

The following proposition is the Capelli-type identity for determinant of $(\hat{q},1)$-$Manin$ matrix.

\begin{proposition}\label{det capelli}
		Let $M$ be an $n\times m$   $(\hat{q},1)$-$Manin$ matrix and $N$ an $m \times s$ matrix and suppose that
		\begin{equation}\label{Capelli_rela1}
	 M_2N_1-N_1M_2=-H_2  P_{s \times m},
		\end{equation}
where $H$ is an $n\times s$ matrix.
In terms of entries this relation can be written as
	\begin{equation}\label{Capelli_rela2}
		 [M_{ij},N_{kl}]=-\delta_{jk}h_{il}, \quad \quad 1\leq i \leq n,1\leq j,k \leq m , 1\leq l \leq s.
	\end{equation}
Let $I=(1\leq i_1 < i_2 < \cdots <i_r \leq n)$ and $K=(k_1,\cdots,k_r)$ be two multi-indices. Then
\begin{equation}
	 	{\cdet}_{\hat{q}}\big((MN)_{IK}+H_{IK} \diag(r-1,r-2,... ,1,0) \big) \\
=\sum\limits_{J} {\cdet}_{\hat{q}}(M_{IJ}){\cdet}(N_{JK})
 \end{equation}
for $r \leq m$, where the sum is taken over all multi-indices of increasing integers $J \subset (1,\dots,m)$. In particular, if $m=n=s$, then
\begin{equation}
		{\cdet}_{\hat{q}}\big(MN+H\diag(n-1,n-2,... ,1,0)\big)={\cdet}_{\hat{q}}(M){\cdet}(N).
\end{equation}
\end{proposition}
\begin{proof}
It follows from relation
	\eqref{Capelli_rela1} that
	\begin{equation}\label{Relation M H}
		(1-P_{\hat{q}})(M_1H_2+H_1M_2 P_{m \times s}) =0.
	\end{equation}
Indeed, let $\Phi=\Psi M$ and $\Phi'=\Psi H$. Relation \eqref{Relation M H} is equivalent to $\phi_{i}\phi_{j}'=-\phi_{j}'\phi_{i}$ for $1\leq i\leq m$, $1\leq j\leq s$.
By relation \eqref{Capelli_rela1} we deduce that $[N_{kl},\phi_{j}]=\delta_{jk}\phi_{l}'$.
Then $\phi_{j}'=[N_{jj},\phi_{j}]$. Thus we have that
\begin{equation*}
			\begin{aligned}
	\phi_{i}\phi_{j}'+ \phi_{j}'\phi_{i} &= \phi_{i}[N_{jj},\phi_{j}]+[N_{jj},\phi_{j}]\phi_{i} \\
	&=\phi_{i}N_{jj}\phi_{j}-\phi_{i}\phi_{j}N_{jj}+N_{jj}\phi_{j}\phi_{i}-\phi_{j}N_{jj}\phi_{i}\\
	&=N_{jj}\phi_{i}\phi_{j}+\phi_{j}\phi_{i}N_{jj}-N_{jj}\phi_{i}\phi_{j}-\phi_{j}\phi_{i}N_{jj}\\
	&=0.
    \end{aligned}
\end{equation*}
For $i=j$,  one has that $\phi_{i}^2=0$. Therefore,
\begin{equation}
\phi_{i}\phi_{i}'+ \phi_{i}'\phi_{i}=\phi_{i}[N_{ii},\phi_{i}]+[N_{ii},\phi_i]\phi_{i}=0.
\end{equation}

The coefficient of $e_{i_1}\otimes\cdots \otimes e_{i_r}$ in
\begin{equation}
A_{\hat{q}}^{(r)}(MN+(r-1)H)_1\cdots (MN)_r e_{k_1}\otimes\cdots \otimes e_{k_r}
\end{equation}
is $\dfrac{1}{r!}{\cdet}_{\hat{q}}((MN)_{IK}+H_{IK}\diag(r-1,r-2,... ,1,0))$. And the coefficient of $e_{i_1}\otimes\cdots \otimes e_{i_r}$ in
\begin{equation}
	A_{\hat{q}}^{(r)}M_1\cdots M_r N_1\cdots N_r e_{k_1}\otimes\cdots \otimes e_{k_r}
\end{equation}
is 	 $\dfrac{1}{r!}\sum\limits_{J} {\cdet}_{\hat{q}}(M_{IJ}){\cdet}(N_{JK})$
, where the sum is over all multi-indices of increasing integers $J \subset (1,\dots,m)$ and $|J|=r$  .

Therefore, it is sufficient to show that
\begin{equation}\label{eq capelli}
	A_{\hat{q}}^{(r)}(MN+(r-1)H)_1\cdots (MN)_r =A_{\hat{q}}^{(r)}M_1\cdots M_r N_1\cdots N_r.
\end{equation}

We prove equation \eqref{eq capelli} by induction on r.
It is obvious for $r=1$. Let
$A_{\hat{q}}^{(r-1)}$  be the $\hat q$-antisymmetrizer on the indices $\{2,\ldots,r\}$. Using the relation $A_{\hat{q}}^{(r)}=A_{\hat{q}}^{(r)}A_{\hat{q}}^{(r-1)}$, we have that
\begin{equation}
	\begin{aligned}
        &A_{\hat{q}}^{(r)}(MN+(r-1)H)_1\cdots (MN)_r \\
        &=A_{\hat{q}}^{(r)}(MN+(r-1)H)_1 A_{\hat{q}}^{(r-1)}(MN+(r-2)H)_2\cdots (MN)_r.
	\end{aligned}
\end{equation}
By induction hypothesis, we have that
\begin{equation}
	\begin{aligned}
	&A_{\hat{q}}^{(r-1)}(MN+(r-2)H)_2\cdots (MN)_r
=A_{\hat{q}}^{(r-1)}M_2\cdots M_r N_2\cdots N_r.
	\end{aligned}
\end{equation}
Using  relations \eqref{Capelli_rela1}, we have
\begin{equation}
	\begin{aligned}
		&A_{\hat{q}}^{(r)}(MN+(r-1)H)_1\cdots (MN)_r \\
		&=A_{\hat{q}}^{(r)}M_1\cdots M_r N_1\cdots N_r\\
		&\quad +A_{\hat{q}}^{(r)}\sum_{i=2}  ^{r}
M_1 M_2 \cdots M_{i-1} H_{i} 	P_{s \times m }^{(1,i)}M_{i+1} \cdots M_{r}N_{2}\cdots N_{r} \\
		&\quad +(r-1)A_{\hat{q}}^{(r)} H_1 M_2\cdots M_r N_2\cdots N_r
	\end{aligned}
\end{equation}
We will prove equation \eqref{eq capelli} by showing that for any $2\leq i\leq r$,
\begin{equation}
	\begin{aligned}
		&A_{\hat{q}}^{(r)}
M_1 \cdots M_{i-1} H_{i} P_{s \times m }^{(1,i)}M_{i+1} \cdots M_{r} \\
	&= -A_{\hat{q}}^{(r)} H_1 M_2\cdots M_r  .
	\end{aligned}
\end{equation}

For any $2\leq i\leq r$,
$A_{\hat{q}}^{(r)} =A_{\hat{q}}^{(r)} \frac{1-P_{\hat{q}}^{(i-1,i)}}{2}$.
By relation
\eqref{Relation M H}, we have that
\begin{equation}
 	\begin{aligned}
 		&A_{\hat{q}}^{(r)}
M_1 \cdots M_{i-1} H_{i}  P_{s \times m}^{(1,i)}M_{i+1} \cdots M_{r} \\
 		&=A_{\hat{q}}^{(r)}\frac{1-P_{\hat{q}}^{(i-1,i)}}
{2} M_1 \cdots M_{i-1} H_{i}	 P_{s \times m}^{(1,i)}M_{i+1} \cdots M_{r} \\
        &=-A_{\hat{q}}^{(r)}\frac{1-P_{\hat{q}}^{(i-1,i)}
}{ 2} M_1 \cdots H_{i-1} M_{i}  P_{m \times
s}^{(i-1,i)}P_{s \times m }^{(1,i)} M_{i+1} \cdots M_{r} \\
        &= -A_{\hat{q}}^{(r)}
M_1 \cdots H_{i-1} M_{i} P_{m \times s}^{(i-1,i)}P_{s \times m}^{(1,i)} M_{i+1} \cdots M_{r} \\
 		&=\cdots\cdots \\
 		&= (-1)^{i-1} A_{\hat{q}}^{(r)}  H_{1}M_2 \cdots
M_{i}	 P_{m \times s}^{(1,2)}P_{m \times s}^{(2,3)}\cdots P_{m \times s}^{(i-1,i)} P_{s \times m}^{(1,i)}	
M_{i+1}\cdots M_{r}. \\
 	\end{aligned}
 \end{equation}
In $\mathrm{End}(\mathbb{C}^{s}\ot {\mathbb{C}^{m}}^{\ot (i-1)}),$  we have
\begin{equation}
	\begin{aligned}
	P_{m \times s}^{(1,2)}P_{m \times s}^{(2,3)}\cdots P_{m \times s}^{(i-1,i)} P_{s \times m}^{(1,i)}=P_{m \times m}^{(2,3)}P_{m \times m}^{(3,4)}\cdots P_{m \times m}^{(i-1,i)}.
	\end{aligned}
\end{equation}
In the following, we denote $P_{m \times m}^{(i,j)}$ by $P^{(i,j)}$.
Thus,
 \begin{equation}
 	\begin{aligned}
 	&A_{\hat{q}}^{(r)}
M_1 \cdots M_{i-1} H_{i} P_{s \times m}^{(1,i)} M_{i+1} \cdots M_{r} \\
 		&= (-1)^{i-1} A_{\hat{q}}^{(r)}  H_{1}M_2
\cdots  M_{i}	 P^{(2,3)}P^{(3,4)}\cdots P^{(i-1,i)} M_{i+1}\cdots M_{r}\\
 		&= (-1)^{i-1} (-1)^{i-2} A_{\hat{q}}^{(r)}
H_{1}M_2 \cdots  M_{i} M_{i+1}\cdots M_{r}\\
 		&=- A_{\hat{q}}^{(r)}  H_{1}M_2 \cdots
M_{r}.\label{capplli-term3}
 	\end{aligned}
 \end{equation}

This completes the proof of the proposition.\end{proof}

%

We can also define the multi-parametric quantum   row  determinants and the $\hat{q}$-minor row  determinants. Let $I=(i_1,\cdots,i_r)$, $J=(j_1,\cdots,j_r)$ be two multi-indices.
We define the multi-parametric  $\hat{q}$-minor row determinant as follows:
\begin{equation}
	{\rdet}_{\hat{q}}(M_{IJ})=\sum \limits_{\sigma \in S_r}  \varepsilon(\hat{q},J,\si)   M_{i_1,j_{\si(1)}}\cdots M_{i_r,j_{\si(r)}}.
\end{equation}
For $m=n$ the row-determinant is defined as $n$-minor row determinant.

\begin{lemma}\label{transpose}
Let $N^{t}$ an $s \times m$ $(\hat{q},\hat{p})$-$Manin$ matrix, then
\begin{align} \label{N eq}
&S_{\hat{p'}} N_1 N_2 A_{\hat{q'}}=0,\\ \label{N eq 1}
&  N_1 \cdots N_r A_{\hat{q'}}^{(r)}=A_{\hat{p'}}^{(r)} N_1 \cdots N_r A_{\hat{q'}}^{(r)},\\ \label{N eq 2}
& S_{\hat{p'}}^{(r)} N_1 \cdots N_r   =S_{\hat{p'}}^{(r)} N_1 \cdots N_r S_{\hat{q'}}^{(r)}.
\end{align}
\end{lemma}

\begin{proof}

Since $N^{t}$ an $s \times m$ $(\hat{q},\hat{p})$-$Manin$ matrix,
then
$A_{\hat{q}} N^{t}_1 N^{t}_2 S_{\hat{p}}=0$. Taking transpose we obtain that
	\begin{equation}
		S_{\hat{p'}} N_1 N_2 A_{\hat{q'}}=0.
	\end{equation}
The Equations \eqref{N eq 1} and \eqref{N eq 2} follows from \eqref{N eq}.
\end{proof}

Here we have the analog  of Proposition \ref{det capelli}:
\begin{proposition}	\label{det capelli2}
Let $M$ be an $n\times m$  matrix and $N^{t}$ an $s \times m$ 	 $(\hat{q},1)$-$Manin$ matrix and suppose that
		\begin{equation} \label{rdet capp rela1}
		 M_2N_1-N_1M_2=-P_{n \times m}H_1.
		\end{equation}
In terms of entries this relation can be written as
\begin{equation}
			[M_{ij},N_{kl}]=-\delta_{jk}h_{il}, \quad \quad 1\leq i \leq n,1\leq j,k \leq m , 1\leq l \leq s.
\end{equation}
		Let $I=(1\leq i_1 < i_2 < \cdots <i_r \leq n)$ and $K=(k_1,\cdots,k_r)$ be two multi-indices. Then
		\begin{equation*}
		{\rdet}_{\hat{q}}\left((MN)_{IK}+\diag(0,1,... ,r-1)H_{IK}\right) = \sum\limits_{J} {\rdet}({M_{IJ}}){\rdet}_{\hat{q}}({N_{JK}}),
		\end{equation*}
		for $r \leq m$, where the sum is taken over all multi-indices of increasing integers $J \subset (1,\dots,m)$.
\end{proposition}
\begin{proof}
It follows from the relation
	\eqref{rdet capp rela1} that
\begin{equation}\label{rdet Relation M H}
	 ( P_{m \times n}N_1H_2 + H_1N_2 )(1-P_{\hat{q'}}) =0.
\end{equation}
Then the proposition can be proved by  Lemma \ref{transpose} and the arguments used in  Proposition \ref{det capelli}.
\end{proof}

\begin{remark}
If $p_{ij}=q_{kl}=1$ for all possible $i,j,k,l$, Proposition \ref{det capelli} and \ref{det capelli2} specialise to the Capelli-type identities for determinants of Manin matrix (\cite{CSS,CFR}).
\end{remark}

In the following, we will give the Capelli-type identities for permanents.
We define
\begin{equation*}
      \mu(\hat{p},\si)= \prod \limits_{\substack{i < j \\ \si(i)>\si(j)}} p_{\si(j),\si(i)}.
\end{equation*}
Similarly,
\begin{equation}
\mu(\hat{p},J,\si)=\prod_{\substack{s<t\\\si_s>\si_t}}
p_{j_{\sigma_{t}}j_{\sigma_s}} .
\end{equation}

	The multi-parametric quantum row permanent of a matrix $M$ is defined as:
	\begin{equation}
		\rper_{\hat{p}}(M)=\sum \limits_{\sigma \in S_n} \mu(\hat{p},\si)   M_{1,\si(1)}\cdots M_{n,\si(n)}.
	\end{equation}

Let $I=(i_1,i_2,\cdots,i_r)$ be any multi-index and $J=(j_1 \leq j_2 \leq \cdots \leq j_r)$ be a multi-index of non-decreasing positive integers. The $\hat{p}$-minor row permanent is defined as
\begin{equation}
	\rper_{\hat{p}}(M_{IJ})=\sum \limits_{\sigma \in S_n} \mu(\hat{p},J,\si)   M_{i_1,j_{\si(1)}}\cdots M_{i_r,j_{\si(n)}}.
\end{equation}
And define the map $\alpha_{J}: [r] \rightarrow [n]$ by $ \alpha_{J}(k)=j_k$ for $1 \leq k \leq r$. Let $v(\alpha_{J})=|\alpha^{-1}(1) |! |\alpha^{-1}(2)| ! \cdots | \alpha^{-1}(n) | !$. The normalized multi-parametric quantum permanents is defined as
\begin{equation}
	\widehat{\rper}_{\hat{p}}(M_{IJ})=\dfrac{1}{v(\alpha_{J})}\rper_{\hat{p}}(M_{IJ}).
\end{equation}

The following proposition is an analog of Cauchy-Binet's formula for permanents of multiparametric Manin matrices.

\begin{proposition}\label{CB permanent}
	Let $M$ be an $n\times m$ matrix and $N$ an $m \times s$ $(\hat{q},\hat{p})$-$Manin$  matrix such that the $N_{ij}$ commute with $M_{kl}$ and $x_{t}$ for all possible indices $i,j,k,l,t$. Let $I=(i_1, \cdots i_r )$ and $K=(k_1 \leq \cdots \leq k_r)$ be two multi-indices. Then
	\begin{equation}
		\widehat{\rper}_{\hat{p}}(MN)_{IK} = \sum\limits_{J} \widehat{\rper}_{\hat{q}}(M_{IJ})\widehat{\rper}_{\hat{p}}(N_{JK}),
	\end{equation}
 where the sum is taken over all multi-indices of non-decreasing integers $J \subset (1,\dots,m)$. In particular, if $m=n=s$, then
	\begin{equation}
	    \widehat{\rper}_{\hat{p}}(MN)=\sum\limits_{J} \widehat{\rper}_{\hat{q}}(M_{[n]J})\widehat{\rper}_{\hat{p}}(N_{J[n]}),
	\end{equation}
 where the sum is taken over all multi-indices of non-decreasing integers $J \subset (1,\dots,m)$.
\end{proposition}
	\begin{proof}
		Let $y_i=\sum_{k=1}^{s}N_{ik}x_k$, $\xi_{i}= \sum_{k=1}^{s} (MN)_{ik}x_{k}$. Then
	\begin{equation}
		y_j y_i = q_{ij}y_i y_j  \quad 1 \leq i,j \leq m,
	\end{equation}
and
   \begin{equation}
	\xi_{i}= \sum \limits_{k=1}\limits^{s} (MN)_{ik}x_{k} = \sum \limits_{k=1}\limits^{s} \sum \limits_{j=1}\limits^{m} M_{ij}N_{jk} x_{k} = \sum \limits_{j=1}\limits^{m} M_{ij} y_{j}.
   \end{equation}
Therefore, we have that
\begin{equation}
   	\begin{aligned}
   	\xi_{i_{1}}\cdots \xi_{i_{r}}&= \sum \limits_{k_1,\dots,k_r=1}\limits^{s}(MN)_{i_1k_1}\cdots (MN)_{i_rk_r}x_{k_1}\cdots x_{k_r}\\
   	&=\sum \limits_{1\leq k_1\leq \dots\leq k_r\leq s}\widehat{\rper}_{\hat{p}}((MN)_{IK})x_{k_1}\cdots x_{k_r}.
   	\end{aligned}
\end{equation}

On the other hand,
\begin{equation}
	\begin{aligned}
		\xi_{i_{1}}\cdots \xi_{i_{r}}&= \sum \limits_{j_1,\dots,j_r=1}\limits^{m} M_{i_1j_1}\cdots M_{i_rj_r}y_{j_1}\cdots y_{j_r}\\
		&=\sum \limits_{1\leq j_1\leq \dots\leq j_r\leq m}\widehat{\rper}_{\hat{q}}(M_{IJ})y_{j_1}\cdots y_{j_r}\\
		&=\sum \limits_{\substack{1\leq j_1\leq \dots\leq j_r\leq m \\ 1\leq k_1\leq \dots\leq k_r\leq s}}\widehat{\rper}_{\hat{q}}(M_{IJ}) \widehat{\rper}_{\hat{p}}(N_{JK})x_{k_1}\cdots x_{k_r}.
	\end{aligned}
\end{equation}
Comparing the coefficients of $x_{k_1}\cdots x_{k_r}$, one has the proposition.
\end{proof}

The following proposition is an analogue of Capelli-type identity for permanents of $(1,\hat p)$-$Manin$ matrices.

\begin{proposition}\label{Capelli-type identity for permanent}
	Let $M$ be an $n\times m$ matrix and $N$ an $m \times s$  $(1,\hat{p})$-$Manin$  matrix such that
	\begin{equation}
	 M_2N_1-N_1M_2=-H_2P_{s\times m}.   \label{per capp rel1}
	\end{equation}
If $m=1$, suppose further that
\begin{equation}
(P_{m\times n}N_1H_2-H_1N_2)(1 + P_{\hat{p}}) = 0.
\end{equation}
Let $I=(i_1,\cdots , i_r)$  and $K=(1\leq k_1 \leq k_2 \leq \cdots \leq k_r \leq s)$ be two multi-indices. Then
	\begin{equation}
		\widehat{\rper}_{\hat{p}}((MN)_{IK}-diag(0,1,\dots,r-1)H_{IK} )= \sum\limits_{J} \widehat{\rper}(M_{IJ})\widehat{\rper}_{\hat{p}}(N_{JK}),\label{per cape 1}
	\end{equation}
	where the sum is taken over all multi-indices of non-decreasing integers $J \subset (1,\dots,m)$. In particular, if $m=n=s$, then
	\begin{equation}\label{e:Cap1}		
	\widehat{\rper}_{\hat{p}}(MN-diag(0,1,\dots,r-1)H )= \sum\limits_{J} 	\widehat{\rper}(M_{[n]J})\widehat{\rper}_{\hat{p}}(N_{J[n]}),
	\end{equation}
where the sum is taken over all multi-indices of non-decreasing integers $J \subset (1,\dots,m)$.
\end{proposition}
\begin{proof} For $m\geq 2$
it follows from \eqref{per capp rel1} that
	\begin{equation}
		( P_{m \times n}N_1H_2-H_1N_2)(1+P_{\hat{p}})=0.  \label{per capp rel2}
	\end{equation}

Indeed, let $Y=NX$ and $Y'=HX$. It is equivalent to prove $y_{i}y_{j}'=y_{j}'y_{i}$ for all $i,j$. Since $[y_k,M_{ij}]=\delta_{jk}y_{i}'$, so we have $y_{j}'=[y_{l},M_{jl}]$ for $l \neq i$, then
\begin{equation*}
	\begin{aligned}
		y_{i}y_{j}'-y_{j}'y_{i} &= [y_{i},[y_{l},M_{jl}]]\\
	&=[[y_{i},y_{l}],M_{jl}]+[y_{l},[y_{i},M_{jl}]]\\
		&=0.
	\end{aligned}
\end{equation*}

The coefficient of $e_{i_1}\otimes\cdots \otimes e_{i_r}$ in
\begin{equation}
	\begin{aligned}
		  (MN)_1\cdots (MN-(r-1)H)_r S_{\hat{p}}^{(r)}e_{k_1}\otimes\cdots \otimes e_{k_r}
	\end{aligned}
\end{equation}
is $ \frac{v(\alpha_{K})}{r!} \widehat{\rper}_{\hat{p}}((MN)_{IK}-diag(0,1,\dots,r-1) H_{IK})$. And the coefficient of $e_{i_1}\otimes\cdots \otimes e_{i_r}$ in
\begin{equation}
      M_1\cdots M_r N_1\cdots N_r S_{\hat{p}}^{(r)}e_{k_1}\otimes\cdots \otimes e_{k_r}
\end{equation}
is $ \frac{v(\alpha_{K})}{r!} \sum\limits_{J}  	\widehat{\rper}(M_{IJ})\widehat{\rper}_{\hat{p}}(N_{JK})$, where the sum is taken over all multi-indices of non-decreasing integers $J \subset (1,\dots,m)$ and $|J|=r$.
Therefore, it is sufficient to show that
\begin{equation}
	\begin{aligned}
	(MN)_1\cdots (MN-(r-1)H)_r S_{\hat{p}}^{(r)}  =M_1\cdots M_r N_1\cdots N_r S_{\hat{p}}^{(r)}  .\label{per cape 1'}
	\end{aligned}
\end{equation}
We will prove equation \eqref{per cape 1'} by induction on r.
It is obvious for $r=1$.
Let
$S_{\hat{p}}^{(r-1)}$  be the $\hat p$-symmetrizer on the indices $\{1,\ldots,r-1\}$. Using the relation $S_{\hat{p}}^{(r)}=S_{\hat{p}}^{(r-1)}S_{\hat{p}}^{(r)}$, we have that

\begin{equation}
	\begin{aligned}
		&(MN)_1\cdots (MN-(r-1)H)_r S_{\hat{p}}^{(r)} \\
		&=(MN)_1\cdots (MN-(r-2)H)_{r-1} S_{\hat{p}}^{(r-1)}   (MN-(r-1)H)_{r}S_{\hat{p}}^{(r)}.
	\end{aligned}
\end{equation}
By the induction hypothesis we have that
\begin{equation}
	\begin{aligned}
		&(MN)_1\cdots (MN-(r-2)H)_{r-1} S_{\hat{p}}^{(r-1)}=M_1\cdots M_{r-1} N_1\cdots N_{r-1} S_{\hat{p}}^{(r-1)}.
	\end{aligned}
\end{equation}

Applying relation \eqref{per capp rel1}, we have that
\begin{equation}
	\begin{aligned}
		&(MN)_1\cdots (MN-(r-1)H)_r S_{\hat{p}}^{(r)}  \\
		&=M_1\cdots M_{r-1} N_1\cdots N_{r-1}
(MN-(r-1)H)_r S_{\hat{p}}^{(r)}   \\
		&=M_1\cdots M_{r} N_1\cdots N_{r}
S_{\hat{p}}^{(r)} \\
		& \quad +\sum_{i=1}^{r-1}M_1\cdots
M_{r-1} N_1\cdots N_{i-1} P_{n \times m  }^{(i,r)} H_{i}N_{i+1}\cdots N_{r}S_{\hat{p}}^{(r)} \\
		&\quad -(r-1)M_1\cdots M_{r-1}
N_1\cdots N_{r-1}H_{r}S_{\hat{p}}^{(r)} .\label{per term1}\\
	\end{aligned}
\end{equation}

In the following we will show that
\begin{equation}
	\begin{aligned}
  N_1\cdots N_{i-1}	 P_{n \times m}^{(i,r)} H_{i}N_{i+1}\cdots N_{r}S_{\hat{p}}^{(r)}=
N_1\cdots N_{r-1}H_{r}S_{\hat{p}}^{(r)}.
	\end{aligned}
\end{equation}
for any $1\leq i\leq r-1.$

In  $\mathrm{End}({\mathbb{C}^{m}}^{\ot (r-i)} \ot \mathbb{C}^{n})$,  we have
	\begin{equation}\label{eq permutaion}
		\begin{aligned}
            P_{n \times
m}^{(i,r)}P_{m \times n}^{(i,i+1)}P_{m \times n}^{(i+1,i+2)}\cdots P_{m \times n}^{(r-1,r)}=
P_{m \times m}^{(i,i+1)}P_{m \times m}^{(i+1,i+2)}\cdots P_{m \times m}^{(r-2,r-1)}.
		\end{aligned}
	\end{equation}

By relations \eqref{per capp rel2} and \eqref{eq permutaion}
we have that
\begin{equation}
	\begin{aligned}
		&N_1\cdots N_{i-1}	 P_{n \times
m}^{(i,r)}H_{i}N_{i+1}\cdots N_{r} S_{\hat{p}}^{(r)}\\
		&= N_1\cdots N_{i-1}
	 P_{n \times
m}^{(i,r)}P_{m \times n}^{(i,i+1)}P_{m \times n}^{(i+1,i+2)}\cdots P_{m \times n}^{(r-1,r)} N_{i}\cdots N_{r-1}H_{r}S_{\hat{p}}^{(r)}\\
		&=  N_1\cdots
N_{i-1}	 P^{(i,i+1)}P^{(i+1,i+2)}\cdots P^{(r-2,r-1)} N_{i}\cdots N_{r-1}H_{r}S_{\hat{p}}^{(r)}\\
		&= N_1\cdots
N_{r-1}H_{r}S_{\hat{p}}^{(r)}.\label{per term2}
	\end{aligned}
\end{equation}

Combining \eqref{per term1} and \eqref{per term2}, we obtain identity \eqref{per cape 1'}.
\end{proof}

\begin{remark}
 Suppose that $p_{ij}=q_{kl}=1$ for all possible $i,j,k,l$,  then \eqref{e:Cap1} is the Capelli-type identity for permanent of Manin matrix  \cite{CFR}.
\end{remark}

We can also define the multi-parametric quantum column permanents and the $\hat{p}$-minor permanents. Let $I=(i_1,\cdots,i_r)$, $J=(j_1,\cdots,j_r)$ be two multi-indices. The $\hat{p}$-minor permanent is defined as:
\begin{equation}
	{\cper}_{\hat{p}}(M_{IJ})=\sum \limits_{\sigma \in S_n}  \mu(\hat{p},I,\si)    M_{i_{\si(1),j_1}}\cdots M_{i_{\si(r),j_r}}.
\end{equation}
The (normalized) column permanent is defined similarly.
Then we have the following analog of Proposition \ref{Capelli-type identity for permanent}.
\begin{proposition}
	Let $M^{t}$ be an $m\times n$ $(1,\hat{p})$-$Manin$ matrix and $N$ an $m \times s$  matrix and suppose that
\begin{equation}\label{cper capp rel1}
	M_2N_1-N_1M_2=-H_2 	P_{s \times m}.
\end{equation}
and if $m=1$ suppose further that
\begin{equation}\label{cper Relation MH}
	(1+P_{\hat{p'}})(M_1H_2-H_1M_2	 P_{m \times s })=0.
\end{equation}
Let $I=(i_1,\cdots,i_r)$  and $K=(1\leq k_1 \leq k_2 \leq \cdots \leq k_r \leq s)$ be two multi-indices. Then
\begin{equation*}
	\widehat{\cper}_{\hat{p}}((MN)_{IK}-H_{IK}diag(r-1,\dots,1,0) )= \sum\limits_{J} \widehat{\cper}_{\hat{p}}(M_{IJ}) \widehat{\cper} (N_{JK}),
\end{equation*}
where the sum is taken over all multi-indices of non-decreasing integers $J \subset (1,\dots,m)$.
\end{proposition}
\begin{proof}
 The  proposition can be proved by  Lemma \ref{transpose} and the arguments used in  Proposition \ref{Capelli-type identity for permanent}.
\end{proof}

\section{MacMahon Theorem}\label{s:macmahon}

Let $A$ be any $n\times n$ complex matrix, and $x_1,\ldots,x_n$ a set of variables.
Denote by $G(k_1,\ldots k_n)$ the coefficient of $x_1^{k_1}\ldots x_n^{k_n}$ in
\begin{equation*}
 \prod_{i=1}^{n}(a_{i1}x_1+\cdots+a_{in}x_n)^{k_i}.
\end{equation*}
Let $t_1,\ldots,t_n$ be another set of variables, $T=diag(t_1,\ldots,t_n)$.
Then  the  MacMahon master theorem holds \cite{Mac}:
\begin{equation*}
  \sum_{k_1,\ldots,k_n}G(k_1,\ldots k_n)t_1^{k_1}\cdots t_n^{k_n}=\frac{1}{\det(I-TA)}.
\end{equation*}
By taking $t_1=t_2=\cdots=t_n=t$,
\begin{equation*}
  \sum_{k_1+\ldots+k_n=k}G(k_1,\ldots k_n)t^k =\frac{1}{\det(I-tA)}.
\end{equation*}

Generalizations of the MacMahon master theorem have been obtained for one-parameter $q$-Manin matrices \cite{GLZ, KP} and super Manin matrices \cite{MR}.
In this section, we study the MacMahon Master Theorem for general $A$-Manin matrices. As a special case, we derive the MacMahon Master Theorem for multiparametric quantum Manin matrices.

Let $A $ be any idempotent, we call the $(A,A)$-Manin matrix $M$ an  $A$-Manin matrix.
Let $P_{A}=1-2A$, then $P_{A}^2=1$. We assume that $P_A$ satisfy the braid relation:
\begin{equation*}
P_{A}^{(12)}P_{A}^{(23)}P_{A}^{(12)}=P_{A}^{(23)}P_{A}^{(12)}P_{A}^{(23)}.
\end{equation*}
By $S^{(k)}$ and $A^{(k)}$ we denote the respective the normalized symmetrizer and antisymmetrizer:
\begin{equation*}
S^{(k)}=\frac{1}{k!}\sum_{\sigma\in S_k}  P_A^{\sigma}, \qquad
A^{(k)}=\frac{1}{k!}\sum_{\sigma\in S_k}\text{sgn}\ \sigma \cdot P_A^{\sigma}
\end{equation*}

We set
\begin{align}
\text{Bos}&=1+\sum_{k=1}^{\infty}tr_{1,\ldots,k}S^{(k)}M_1\cdots M_k,\\
\text{Ferm}&=1+\sum_{k=1}^{\infty}(-1)^ktr_{1,\ldots,k}A^{(k)}M_1\cdots M_k.
\end{align}

Using the R-matrix method \cite{MR}, we have the generalized MacMahon Theorem.
\begin{theorem} For any $A$-Manin matrix,
one has that
\begin{equation}
\text{Bos}\times \text{Ferm} =1.
\end{equation}
\end{theorem}

\begin{proof}
We have that $A^{\{r+1,\ldots,k\}}M_1\cdots M_r=M_1\cdots M_r A^{\{r+1,\ldots,k\}},$ where $A^{\{r+1,\ldots,k\}}$ denotes the antisymmetrizer over the copies of $End(\mathbb{C}^n)$ labeled by $\{r+1,\ldots,k\}$. Therefore, it is sufficient to show that
\begin{equation}
\sum_{r=0}^k (-1)^r tr_{1,\ldots,k}S^{(r)}A^{\{r+1,\ldots,k\}}M_1\cdots M_k=0.
\end{equation}

In the following we show that
\begin{equation}\label{trace replacement}
\begin{aligned}
&tr_{1,\ldots,k}S^{(r)}A^{\{r+1,\ldots,k\}}M_1\cdots M_k\\
=&tr_{1,\ldots,k}\frac{r(k-r+1)}{k}S^{(r)}A^{\{r,\ldots,k\}}
M_1\cdots M_k\\
+
&tr_{1,\ldots,k}\frac{(r+1)(k-r)}{k}S^{(r+1)}A^{\{r+1,\ldots,k\}}
M_1\cdots M_k\\
\end{aligned}
\end{equation}
By the relations of elements in the group algebra of $S^{(k)}$, we have
\begin{equation}
\begin{aligned}
&(k-r+1)A^{\{r,\ldots,k\}}=A^{\{r+1,\ldots,k\}}-(k-r)A^{\{r+1,\ldots,k\}}  P_{A}^{(r,r+1)} A^{\{r+1,\ldots,k\}},\\
&(r+1)S^{(r+1)}=S^{(r)}+rS^{(r)} P_{A}^{(r,r+1)} S^{(r)}.
\end{aligned}
\end{equation}

\begin{equation}
\begin{aligned}
&tr_{1,\ldots,k}S^{(r)}A^{\{r+1,\ldots,k\}}  P_{A}^{(r,r+1)} A^{\{r+1,\ldots,k\}}M_1\cdots M_k\\
=&tr_{1,\ldots,k} S^{(r)}  P_{A}^{(r,r+1)} A^{\{r+1,\ldots,k\}}M_1\cdots M_kA^{\{r+1,\ldots,k\}}\\
=&tr_{1,\ldots,k} S^{(r)}  P_{A}^{(r,r+1)} A^{\{r+1,\ldots,k\}}M_1\cdots M_k
\end{aligned}
\end{equation}
Similarly,
\begin{align}
tr_{1,\ldots,k}S^{(r)} P_{A}^{(r,r+1) }S^{(r)} A^{\{r+1,\ldots,k\}}M_1\cdots M_k\\
=tr_{1,\ldots,k} S^{(r)}  P_{A}^{(r,r+1)}A^{\{r+1,\ldots,k\}}M_1\cdots M_k
\end{align}
These imply equation \eqref{trace replacement}. Therefore the telescoping sum equals to zero.
\end{proof}

\begin{remark}
Since the operator
$P_{\hat q}$ satisfies the braid relation, the MacMahon theorem holds for
$(\hat q)$-Manin matrix ($(\hat q, \hat q)$-Manin matrix) as a special case of the Theorem.
\end{remark}

Let $M$ be an $n\times n$ $A$-Manin matrix. Introduce the following sequential partial traces
on the tensor space by $e_0=1$, and for $k\geq 1$
\begin{equation}\label{e:elementary}
e_k=tr_{1,\ldots,k}A^{(k)}M_1\cdots M_k.
\end{equation}
These elements are {\it quantum elementary symmetric functions}.

Let $B, C$ be $n\times n$ matrices.
Denote $B*C=tr_1 P_{A} B_1C_2$.
Let $M^{[k]}$ be the $k$th power of $M$ under the multiplication $*$, i.e.
\begin{align}\label{e:power-m}
M^{[0]}=1, M^{[1]}=M, M^{[k]}=M^{[k-1]}*M,k>1.
\end{align}

\begin{lemma}\label{lemm newton}
Let $M$ be an $n\times n$   $A$-Manin matrix. Then

\begin{equation}\label{lemma newton eq}
\begin{split}
 ktr_{1,\ldots,k-1}A^{(k)}M_1\cdots M_k
=
\sum_{i=0}^{k-1}(-1)^{k+i+1} e_i M^{[k-i]}
\end{split}
\end{equation}
\end{lemma}

\begin{proof}

Using the relation
\begin{equation}
A^{(k)}=\frac{1}{k}A^{(k-1)}- \frac{k-1}{k}A^{(k-1)}  P_{A}^{(k-1,k)} A^{(k-1)},
\end{equation}
we have
\begin{equation}
\begin{split}
&ktr_{1,\ldots,k-1}A^{(k)}M_1\cdots M_k\\
&=e_{k-1}M_k
- (k-1)tr_{1,\ldots,k-1} A^{(k-1)} P_{A}^{(k-1,k)}A^{(k-1)} M_1\cdots M_k\\
&=e_{k-1}M
- (k-1)tr_{1,\ldots,k-2} (A^{(k-1)} M_1\cdots M_{k-1})*M.\\
\end{split}
\end{equation}
 Iterating it we obtain
\begin{equation}\label{eq cayleyhamiton}
ktr_{1,\ldots,k-1}A^{(k)}M_1\cdots M_k=\sum_{i=0}^{k-1}(-1)^{k+i+1} e_i M^{[k-i]}.
\end{equation}
\end{proof}

We denote $A(t)$, $S(t)$ and $T(t)$ as follows:
\begin{align}
S(t)&=\sum_{k=0}^{\infty}t^k trS^{(k)}M_1\cdots M_k,\\
A(t)&=\sum_{k=0}^{\infty}(-t)^k trA^{(k)}M_1\cdots M_k,\\
T(t)&=\sum_{k=0}^{\infty}t^k trM^{[k+1]}.
\end{align}

We now have the generalized Newton's identities for multi-parametric quantum groups.
\begin{theorem}[Newton's identities]
Let $M$ be an $n\times n$ $A$-Manin matrix.  Then
\begin{align}
&\partial_t A(t)=-A(t) T(t),\\
&\partial_t S(t)= T(t)S(t).
 \end{align}
\end{theorem}

\begin{proof}
Take the trace in the equation \eqref{lemma newton eq}, we have
\begin{equation}\label{eq cayleyhamiton}
ke_k=\sum_{i=0}^{k-1}(-1)^{k+i+1} e_i tr M^{[k-i]}.
\end{equation}
This implies that $\partial_t A(t)=-A(t) T(t)$.

It follows from the MacMahon Theorem that
$A(t)S(t)=S(t)A(t)=1$.
By the Leibniz rule, we have
\begin{align}
(\partial_t A(t))  S(t) +A(t)\partial_t S(t)=0.
\end{align}

Therefore, $A(t)\partial_t S(t)=A(t) T(t)S(t)$. Multiplying the both sides of the equation by $S(t)$ from the left, we have
\begin{equation}
\partial_t S(t)= T(t)S(t).
\end{equation}
\end{proof}

\section{Cayley-Hamiton's theorem}\label{s:cayley-hamilton}

The Cayley-Hamilton theorem states that every square matrix over a commutative ring satisfies its own characteristic equation.
This theorem was generalized to quantum semigroups $\mathrm{M}_q(n)$ \cite{JJZ}. In this section we study the Cayley-Hamilton theorem for multiparameter Manin matrices.

Let $M$ be an $n\times n$ $(\hat q)$-Manin matrix.
The $\hat q$-characteristic polynomial of $M$ is defined as
\begin{equation}
\mathrm{char}_{\hat q}(M,t)=\sum_{k=0}^n(-1)^k e_k t^{n-k}.
\end{equation}
where the $e_k$ are the quantum elementary symmetric functions \eqref{e:elementary}.
We have the following Cayley-Hamiton Theorem.
\begin{theorem}
Let $M$ be an $n\times n$ $(\hat q)$-Manin matrix. Then
\begin{equation}
\sum_{k=0}^n(-1)^k e_k M^{[n-k]}=0.
\end{equation}
\end{theorem}

\begin{proof}
Since $tr_{1,\ldots,n-1}A^{(n)}=\frac{1}{n}$,
\begin{equation}
tr_{1,\ldots,n-1}A^{(n)}M_1\cdots M_n=\frac{1}{n}{\cdet}_{\hat{q}}(M)=\frac{1}{n}e_n.
\end{equation}
Then the theorem follows from Lemma \ref{lemm newton}.
\end{proof}

Let $M$ be an $n\times m$ $(\hat{q},\hat{p})$-Manin matrix  and $N$
be an $m\times n$ $(\hat{p},\hat{q})$-Manin matrix
such that the $N_{ij}$  commute with $M_{kl}$ and for all possible indices  $i,j,k,l$.
It follows from \cite{S} that $MN$ is an $n\times n$ $(\hat{q})$-Manin matrix
and $NM$ is an $m\times m$ $(\hat{p})$-Manin matrix.

\begin{theorem}
The characteristic polynomials of Manin matrices $MN$ and $NM$ satisfy the relation
\begin{align}
\mathrm{char}_{\hat q}(MN,t)=t^{n-m}\mathrm{char}_{\hat p}(NM,t).
\end{align}
 \end{theorem}
\begin{proof}
The  characteristic polynomial of $MN$ and $NM$ are
\begin{align}
\mathrm{char}_{\hat q}(MN,t)=\sum_{k=0}^n(-1)^k e_k(MN) t^{n-k},\\
\mathrm{char}_{\hat p}(NM,t)=\sum_{k=0}^m(-1)^k e_k(NM) t^{m-k},
\end{align}

Without loss of generality, we assume that $n\leq m$.
It follows from the Cauchy-Binet  formula that
$
e_k(NM)=0$ for $k>n$. For $k\leq n$,

\begin{equation}
\begin{split}
e_k(MN)&=\sum_{I\subset[1,n]} {\cdet}_{\hat{q}}((MN)_{II})\\
&=\sum_{I\subset[1,n],J\subset[1,m]}{\cdet}_{\hat{q}}({M_{IJ}})
{\cdet}_{\hat{p}}({N_{JI}})\\
&=\sum_{I\subset[1,n],J\subset[1,m]}{\cdet}_{\hat{p}}({N_{JI}})
{\cdet}_{\hat{q}}{M_{IJ}})\\
&=\sum_{J\subset[1,m]} {\cdet}_{\hat{q}}((NM)_{JJ})=e_k(NM).
 \end{split}
\end{equation}
This completes the proof.
\end{proof}

\section{The multiparameter quantum  Yangian}\label{s:multiYangian}

The $q$-Yangian algebra \cite{Ch, NT} is a subalgebra of the quantum affine algebra $U_q(\widehat{gl}_n)$.
In this section we show that the matrix $T(u)$ of multiparameter quantum Yangian is a multiparameter Manin matrix.
Using the results in the previous sections we obtain some identities for multiparameter quantum Yangians.

Let $p_{ij}, q_{ij}$ $(1\leq i,j\leq n)$ and $u$ be parameters satisfying the following condition:
\begin{align}\label{e:cond-p}
&p_{ij}q_{ij}=u^2, \quad u^2\neq -1, \quad i<j.
\\
&p_{ij}p_{ji}=q_{ij}q_{ji}=1,\\
&p_{ii}=q_{ii}=1.
\end{align}
The multiparameter $R$ matrix $R(z)$ in $\mathrm{End}(\mathbb C^n\otimes \mathbb C^n)\simeq\mathrm{End}(\mathbb C^n)^{\ot 2}$ is given as:
\begin{equation}
\begin{split}
R(z)&=(zu-u^{-1})\sum_{i}e_{ii}\otimes e_{ii}
+\sum_{i< j}(zu^{-1}p_{ij}-uq_{ij}^{-1})e_{ii}\otimes e_{jj}\\
&+\sum_{i> j}(zu^{-1}q_{ji}-up_{ji}^{-1})e_{ii}\otimes e_{jj}
+z(u-u^{-1})\sum_{i<j}e_{ij}\otimes e_{ji}\\
&+(u-u^{-1}  )\sum_{i>j}e_{ij}\otimes e_{ji}
\end{split}
 \end{equation}
It satisfies the Yang-Baxter equation:
\begin{equation}\label{YBE}
R_{12}(z/w)R_{13}(z)R_{23}(w)=R_{23}(w)R_{13}(z)R_{12}(z/w).
\end{equation}

We define the {\it multiparameter quantum Yangian} $Y_{\hat p, u}(\gl_n)$ as the unital associative algebra generated by elements
$t_{ij}^{(r)}$, where $1\leq i,j\leq n$ and $r\geq 0$, such that $t^{(0)}_{ii}$ are invertible
and $ t_{ij}^{(0)}=0$ for $i<j$. These elements satisfies the
following relation
\begin{align}
 \begin{split}
\label{Yangian relation}
 R(z/w)T_1(z)T_2(w)&=T_2(w)T_1(z) R(z/w),
 \end{split}
\end{align}
where $T(z)=(t_{ij}(z))_{1\leq i,j\leq n}$ and
\begin{equation}
t_{ij}(z)=\sum_{r=0}^{\infty}t_{ij}^{(r)}z^{-r}.
\end{equation}

Denote $\hat{R}(z)=PR(z)$, then Eq. \eqref{Yangian relation} is equivalent to
\begin{align}
 \begin{split}
 \hat R(z/w)T_1(z)T_2(w)=T_1(z) T_2(w) \hat R(z/w).
 \end{split}
\end{align}
Therefore,
\begin{equation}\label{e:Ru}
\begin{split}
\hat R(u^{-2})=
(u-u^{-1})\left[\sum_{i<j}\left(E_{ii}\otimes E_{jj}-q_{ij}^{-1}E_{ji}\otimes E_{ij}\right) \right.\\
-\left.\sum_{i<j}p_{ij}^{-1} \left(E_{ij}\otimes E_{ji}-q_{ij}^{-1}E_{jj}\otimes E_{ii}\right) \right].
\end{split}
\end{equation}

We denote by $A^{(k)}$ the normalized antisymmetrizer:
\begin{align}
A^{(k)}=\frac{1}{[k]_{u^2}!}
\sum_{\sigma,\rho \in S_k
\atop  i_1<\cdots<i_k}
 \varepsilon(\hat p, I,\rho) \varepsilon(\hat q,I,\sigma)
e_{i_{\sigma(1)}i_{\rho(1)}}\ot \cdots \ot e_{i_{\sigma(k)}i_{\rho(k)}},
\end{align}
summed over all increasing multi-indices $I=(i_1,\ldots,i_k)$,
then $(A^{(n)})^2=A^{(n)}$.

Let  $R(\lambda_1,\ldots,\lambda_k)$ be the operator in $\mathrm{End}(\mathbb{C}^{n})^{\otimes k}$
defined by
\begin{equation}
R(\lambda_1,\ldots,\lambda_k)=(\hat R_{12}\cdots \hat R_{k-1,k})\cdots (\hat R_{12} \hat R_{23}) \hat R_{12},
\end{equation}
where $\hat R_{ij}=\hat R_{ij}(\lambda_j/\lambda_i)$. It follows from \eqref{e:Ru} that
\begin{equation}
\begin{split}
&R(1,u^{-2},\ldots,u^{-2k+2})\\
&=u^{\frac{k(k-1)}{2}} \prod_{0\leq i<j\leq k-1}(1-u^{2(i-j)})[k]_{u^2}! A^{(k)}.
\end{split}
\end{equation}

The following lemma is easy to verify.
\begin{lemma} We have the identities:
\begin{align}\label{eq antisym 1}
&A_{\hat p}^{(k)}A ^{(k)}=A_{\hat p}^{(k)}\\ \label{eq antisym 2}
&A ^{(k)}A_{\hat p}^{(k)}=A^{(k)}
\end{align}
\end{lemma}

\begin{proposition}
The matrix  $M=T(z)u^{2z\frac{\partial}{\partial_z}}$ is a $\hat p$-Manin matrix.
\end{proposition}

\begin{proof}
Take $w=u^2 z$ we have that
 \begin{equation}
\hat R(u^{-2})T_{1}(z)T_2(u^2z )= T_{1}(z) T_2(u^2z ) \hat R(u^{-2}).
 \end{equation}
Multiplying both sides of
the equation by the operator $u^{4z\frac{\partial}{\partial_z}}$ from the right,
we obtain
\begin{equation}
\hat R(u^{-2})T_{1}(z)u^{2z\frac{\partial}{\partial_z}} T_2(z )u^{2z\frac{\partial}{\partial_z}}
= T_1(z )u^{2z\frac{\partial}{\partial_z}}T_{2}(z)  u^{2z\frac{\partial}{\partial_z}} \hat R(u^{-2}).
\end{equation}
Together with the equations \eqref{eq antisym 1} and \eqref{eq antisym 2},
we have that

\begin{equation}
\begin{split}
A_{\hat p}M_1M_2A_{\hat p}=A_{\hat p}M_1M_2
\end{split}
\end{equation}
\end{proof}

The element $A^{(k)}T_1(z)T_2(u^2z)\cdots T_k(u^{2k-2}z)$ can be written as
\begin{equation}
\sum_{I,J} \text{qdet}(T(z)_{IJ})\otimes e_{i_1 j_1}\otimes \dots \otimes e_{i_mj_m},
\end{equation}
summed over all the multi-indices $I=(i_1,\ldots,i_k)$, $J=(j_1,\ldots,j_k)$.
If $I$ is increasing, $\text{qdet}(T(z)_{IJ})$ can be written as
\begin{equation}
\text{qdet}(T(z)_{IJ})=\sum_{\sigma\in S_k} \varepsilon(\hat p, I,\sigma) t_{i_{\sigma_1}j_1}(z)\cdots t_{i_{\sigma_k}j_k}(u^{2k-2}z).
\end{equation}

In particular, when $I=J=(1,2,\cdots,n)$, it is called the determinant of $T(z)$.
\begin{proposition}
The $\hat p$-determinant of $M=T(z)u^{2z\frac{\partial}{\partial_z}}$ satisfies
\begin{equation}
{\cdet}_{\hat p}M_{IJ}=\qdet(L(z))u^{2kz\frac{\partial}{\partial_z}}
\end{equation}
where $I=(i_1,\ldots ,i_k)$ and $J=(j_1,\ldots ,j_k)$ are multi-index and $I$ is increasing.
\end{proposition}
\begin{proof}  By definition we have that
\begin{equation}
\begin{split}
{\cdet}_{\hat{p}}(M_{IJ})
&=\sum_{\sigma\in S_r}  \varepsilon({\hat{p}},I,\sigma) M_{i_{\sigma(1)},j_1}\cdots M_{i_{\sigma(r)},j_r}\\
&=\sum_{\sigma\in S_r}  \varepsilon({\hat{p}},I,\sigma) t_{i_{\sigma(1)},j_1}(z)\cdots t_{i_{\sigma(r)},j_r}(u^{2k-2}z)u^{2kz\frac{\partial}{\partial_z}}\\
&=\qdet(T(z))u^{2kz\frac{\partial}{\partial_z}}.
\end{split}
\end{equation}

\end{proof}
Let $e_k(T(z))=\sum_{I}\qdet(L(z)_{II})$, where the sum is over all increasing multi-indices $I$.
Recall that $B*C=tr_1P_{\hat p} B_1C_2$.
Let $T(z)^{[k]}$ be the $k$th power of $T(z)$ under the multiplication $*$, i.e.
\begin{align}
T(z) ^{[0]}=1, T(z)^{[k]}=T(z)^{[k-1]}*T(u^{2k-2}z),k>1.
\end{align}

\begin{proposition}
The following is the Cayley-Hamilton theorem for the $\hat p$-Yangian.
\begin{equation}
\sum_{k=0}^n(-1)^k e_k(T(z)) T(u^{2k}z)^{[n-k]}=0.
\end{equation}
\end{proposition}
\begin{proof}
By the  Cayley-Hamilton theorem for the Manin matrix $T(z)u^{2z\frac{\partial}{\partial_z}}$ we have that
\begin{equation}
\sum_{k=0}^n(-1)^k e_k(T(z)) T(u^{2k}z)^{[n-k]}u^{2nz\frac{\partial}{\partial_z}}=0.
\end{equation}
This completes the proof.
\end{proof}

The following result follows from MacMahon's Theorem for the Manin matrix $T(z)u^{2z\frac{\partial}{\partial_z}}$.
\begin{theorem}The following are Muir's identities for the multiparametric $q$-Yangian.
\begin{equation}
\sum_{r=0}^k (-1)^r tr_{1,\ldots,k}S_{\hat p}^{(r)}A_{\hat p}^{\{r+1,\ldots,k\}}
T_1(z)\cdots T_k(u^{2k-2}z)=0.
\end{equation}
\begin{equation}
\sum_{r=0}^k (-1)^r tr_{1,\ldots,k}A_{\hat p}^{(r)}S_{\hat p}^{\{r+1,\ldots,k\}}
T_1(z)\cdots T_k(u^{2k-2}z)=0.
\end{equation}
\end{theorem}

Let $h_k(T(z))=tr_{1,\ldots,k} S_{\hat p}^{(k)} T_1(z)\cdots T_k(u^{2k-2}z)$. Then the following result follows from Newton's identities for  Manin matrices.
\begin{theorem} The following equations are Newton-like identities for the multiparametric quantum Yangian:
\begin{align}
&ke_k(T(z))=\sum_{i=0}^{k-1}(-1)^{k+i+1} e_i(T(z)) tr T(u^{2i}z)^{[k-i]},\\
&kh_k(T(z)) =\sum_{i=1}^{k}  tr T(z)^{[i]}   h_{k-i}(T(u^{2i}z)).
\end{align}
\end{theorem}
\bigskip
\centerline{\bf Acknowledgments}
\medskip
The work is supported in part by the National Natural Science Foundation of China grant nos.
12171303 and 12001218, the Humboldt Foundation, the Simons Foundation grant no. 523868,
and the Fundamental Research Funds for the Central Universities grant no. CCNU22QN002.

\bibliographystyle{amsalpha}

\end{document}